\documentclass[11pt,twoside]{article}
\usepackage{amsmath,amssymb,amsthm}
\usepackage{amsmath,amssymb,amsthm}
\usepackage{graphicx,color,subfigure}
\usepackage{epstopdf}
\usepackage{txfonts}
\usepackage{cite}
\usepackage{caption}
\allowdisplaybreaks[4]

\newtheorem{lm}{Lemma}[section]
\newtheorem{thm}{Theorem}[section]

\newcounter{saveeqn}%

\makeatletter
\evensidemargin
\oddsidemargin
\makeatother

\makeatletter
\@addtoreset{equation}{section}
\makeatother

\title{\Large\bf Grazing-sliding bifurcations in planar $\mathbb{Z}_2$-symmetric Filippov systems \thanks{
Supported by by National Key R\&D Program of China \#2022YFA1005900, NSFC \#12201509 and \#12271378, NSFSPC \#2023NSFSC1360 and Sichuan Science and Technology Program \#2024NSFJQ0008.}
}

\author{Xingwu Chen$^1$, ~~Zhihao Fang$^2$, ~~Tao Li$^3$~\!\!
\footnote{Corresponding author: Tao Li (litao@swufe.edu.cn)}
\\
{\small 1. School of Mathematics, Sichuan University, Chengdu, Sichuan 610064, P. R. China}\\
{\small 2. School of Mathematics, China University of Mining and Technology,}\\
  {\small Xuzhou, Jiangsu 221116, P. R. China}\\
{\small 3. School of Mathematics, Southwestern University of Finance and Economics,}\\
  {\small Chengdu, Sichuan 611130, P. R. China}
}
\date{}

\begin{document}
\maketitle

%%%%%%%%%%%%%%%%%%%%%%%%%%%%%%%%%%%
\begin{abstract}
This paper aims to explore the effect of $\mathbb{Z}_2$-symmetry on grazing-sliding bifurcations in planar Filippov systems. We consider the scenario where the unperturbed system is $\mathbb{Z}_2$-symmetric and its subsystem exhibits a hyperbolic limit cycle grazing the discontinuity boundary at a fold. Employing differential manifold theory, we reveal the intrinsic quantities of unfolding the bifurcation and rigorously demonstrate that the bifurcation set is a codimension-two submanifold of the set of all $\mathbb{Z}_2$-symmetric Filippov systems. After deriving an explicit non-degenerate condition with respect to parameters, we systematically establish the complete bifurcation diagram with exact asymptotics for all bifurcation boundaries by displacement map method combined with asymptotic analysis.

\end{abstract}

\baselineskip 15pt
\parskip 10pt
\thispagestyle{empty}
\setcounter{page}{1}

%%%%%%%%%%%%%%%%%%%%%%%%%%%%%%%%%%%%%%%%%%%%%%%%%%%%%%%%%%%%
\section{Introduction}
\setcounter{equation}{0}
\setcounter{lm}{0}
\setcounter{thm}{0}
\setcounter{rmk}{0}
\setcounter{df}{0}
\setcounter{cor}{0}

The bifurcation theory of dynamical systems has been developed to characterize how the qualitative behavior of systems change under parameter variations. While smooth dynamical systems exhibit well-documented bifurcation phenomena with established analytical frameworks, methods and theories, real-world evolutionary processes often demonstrate intrinsic discontinuity driven by transient events. Typical manifestations include switching mechanisms in circuit systems\cite{MD}, stick-slip oscillations in dry friction systems \cite{MG,ABN}, and threshold-driven population control in pest management \cite{STJLYX}, etc. Such discontinuous dynamical phenomena fundamentally challenge the applicability of smooth systems.
The predominant modelling paradigm is usually the discontinuous piecewise-smooth differential system, also called {\it Filippov system} \cite{MD}. In particular, a planar Filippov system with two zones separated by a smooth curve can be written in the form $(\dot x, \dot y)=Z(x, y)$ with discontinuous piecewise-smooth vector field
\begin{eqnarray}
Z(x, y)=\left\{
\begin{aligned}
  &Z^+(x, y)
~~~~&& {\rm if}~(x, y)\in\Sigma^+,\\
  &Z^-(x, y)
~~~~&& {\rm if}~(x, y)\in\Sigma^-,\\
\end{aligned}
\right.
\label{mainsystem}
\end{eqnarray}
where $Z^\pm\in\mathfrak{X}$ and $\mathfrak{X}$ is the set of all $C^k$ ($k\ge1$) vector fields
defined on $N:=\{(x,y)\in\mathbb{R}^2:x^2+y^2\le r^2\}$ and endowed with the $C^k$-topology
\cite[Chapter 2]{HFFFZ}, $r>0$ is sufficiently large, and
$$\Sigma^+:=\left\{(x, y)\in N: ~h(x, y)>0\right\},~~~~~~~~\Sigma^-:=\left\{(x, y)\in N: ~h(x, y)<0\right\}$$
are two zones split by the smooth curve $\Sigma:=\left\{(x,y)\in N: h(x,y)=0\right\}$. Here $h: N\rightarrow\mathbb{R}$ is a $C^k$ function having $0$ as a regular value.
This shift has prompted substantial researchers to investigate the discontinuity-induced bifurcations in Filippov systems.

Among discontinuity-induced bifurcations, {\it sliding bifurcations} have emerged as a central focus in modern bifurcation theory. These bifurcations characterize the dynamical transitions resulting from the perturbations of {\it tangential periodic orbits} ( as formally defined in Section 2). With the efforts of many researchers, some contributions have been achieved in understanding sliding bifurcations, e.g., \cite{YAK,MG1,AGN,FLMHXZ,LTXC,AC,EFEPFT,NTZ,WHW, FC}. Depending on the topological structures of tangential periodic orbits, sliding bifurcations can be classified into four distinct types: grazing-sliding bifurcations, switching-sliding bifurcations, crossing-sliding bifurcations, and multi-sliding bifurcations, see \cite{YAK, MD}. In this paper we will specially focus on {\it grazing-sliding bifurcations}, which occur when a limit cycle of some subsystem grazes the discontinuity boundary. The prominence stems from their ubiquity in practical applications, including but not limited to mechanical models with dry friction \cite{MBPK1,MG,ABN}, Filippov-type predator-prey models \cite{YAK}, and two-stage population models \cite{STJLYX}. The theoretical significance of grazing-sliding bifurcations lies in their capacity to generate novel dynamical phenomena. For example, the grazing cycle can accumulate a sliding segment and thus becomes a {\it sliding cycle} \cite{YAK}, namely an isolated periodic orbit having a segment that coincides with the discontinuity boundary. It also may bifurcate into a cycle involving multiple loops, some of which involve sliding segments \cite{RSHO}, or into chaotic attractors \cite{MBPK1}. Even, multiple or infinitely many attractors of different types may coexist in grazing-sliding bifurcations, see \cite{ATTRACTOR1,ATTRACTOR2,ATTRACTOR3}.

A comprehensive characterization of bifurcation phenomena requires complete identification of all potential bifurcation scenarios and their corresponding diagrammatic representations. This is extremely challenging in high-dimensional grazing-sliding bifurcations, where complex dynamical structures such as multi-loop periodic orbits and chaotic attractors may merge, see \cite{MBPK1,ATTRACTOR1,ATTRACTOR2,ATTRACTOR3,RSHO,MBPK}. Given these complexities, investigating planar systems provides a fundamental starting point for understanding grazing-sliding bifurcation. In \cite{YAK}, researchers showed two codimension-one grazing-sliding bifurcation diagrams for planar Filippov systems under the non-degenerate conditions: (i) the grazing cycle is hyperbolic, and (ii) the grazing point is a {\it regular-fold}, namely at which
one vector field exhibits quadratic tangency to the discontinuity boundary while the other maintains transversal intersection.
The first one is called the {\it persistence scenario} (cf. \cite{MD}), where the grazing cycle is stable and it bifurcates into either a standard cycle or a sliding cycle, as illustrated in Figure~\ref{codimonegsbper} with critical parameter value $\mu=0$. The second one is called the {\it non-smooth fold scenario} (cf. \cite{MD}), where the grazing cycle is unstable, and either a sliding cycle and a standard cycle simultaneously bifurcate from the grazing cycle, or the grazing cycle disappears and no cycles bifurcate, as depicted in Figure~\ref{codimonegsbnonf}
with critical parameter value $\mu=0$.

Following the work of \cite{YAK}, recent research efforts have extended to deal with high-codimension grazing-sliding bifurcations. The analysis of such bifurcations serves dual purposes: enhancing theoretical understanding of low-codimension bifurcations while addressing practical requirements, as real-world models frequently involve multiple parameters. Building upon the framework of \cite{YAK}, there are two principal approaches for generating high-codimension grazing-sliding bifurcations. The first one is to destroy the condition (i), namely maintaining the regular-fold configuration at the grazing point while allowing the grazing cycle to transition from hyperbolic to non-hyperbolic structure. As demonstrated in \cite{AC} for $n$-dimensional systems, this degeneration produces characteristic geometric signatures: at a generic intersection between the smooth and discontinuity-induced bifurcation curves, another bifurcation curve emerges tangentially to the former. Alternatively, relaxing the condition (ii) provides a second pathway, namely preserving the hyperbolicity of the grazing cycle while degenerating the grazing point from a regular-fold to a higher-order tangent point. Studies in \cite{FLMH} and \cite{FLMHXZ} consider that the grazing point is a {\it fold-fold}, at which both vector fields are quadratically tangent to the discontinuity boundary, establishing the stability criteria of the grazing cycle and deriving lower bounds for the maximum number of limit cycles that bifurcate from the grazing cycle. Subsequent work by \cite{LTXC}, under the same degeneration as in \cite{FLMH,FLMHXZ}, reveals four distinct codimension-two grazing-sliding bifurcations, with associated bifurcation diagrams classified through combined analysis of the types of fold-fold and the internal stability of grazing cycle.
For a near-Hamiltonian system, the number
of crossing limit cycles bifurcating from two grazing cycles connecting a fold-fold is given by using Melnikov functions
in \cite{FLMHsfsf}.

\begin{figure}
  \begin{minipage}[t]{0.33\linewidth}
  \centering
  \includegraphics[width=1.470in]{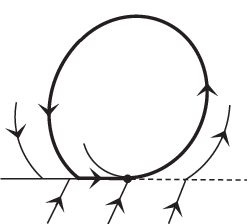}
  \caption*{$\mu<0$}
  \end{minipage}
  \begin{minipage}[t]{0.33\linewidth}
  \centering
  \includegraphics[width=1.50in]{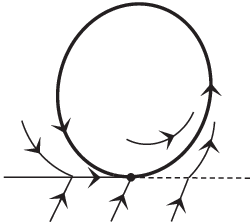}
  \caption*{$\mu=0$}
  \end{minipage}
  \begin{minipage}[t]{0.33\linewidth}
  \centering
  \includegraphics[width=1.52in]{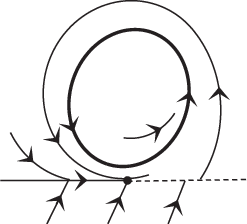}
  \caption*{$\mu>0$}
  \end{minipage}
\caption{{\small Codimension-one grazing-sliding bifurcation for a stable grazing cycle.}}
\label{codimonegsbper}
\end{figure}
\begin{figure}
  \begin{minipage}[t]{0.33\linewidth}
  \centering
  \includegraphics[width=1.50in]{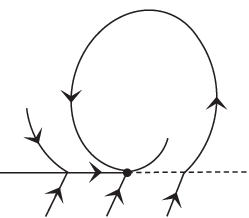}
  \caption*{$\mu<0$}
  \end{minipage}
  \begin{minipage}[t]{0.33\linewidth}
  \centering
  \includegraphics[width=1.470in]{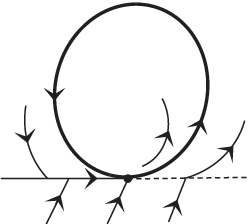}
  \caption*{$\mu=0$}
  \end{minipage}
  \begin{minipage}[t]{0.33\linewidth}
  \centering
  \includegraphics[width=1.50in]{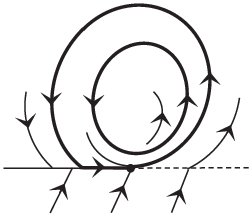}
  \caption*{$\mu>0$}
  \end{minipage}
\caption{{\small Codimension-one grazing-sliding bifurcation for an unstable grazing cycle.}}
\label{codimonegsbnonf}
\end{figure}

In this paper, we investigate grazing-sliding bifurcations in planar $\mathbb{Z}_2$-symmetric Filippov systems. We consider the scenario where the unperturbed system possesses $\mathbb{Z}_2$-symmetry and its subsystem exhibits a hyperbolic limit cycle that grazes the discontinuity boundary at a fold. Under these conditions, the unperturbed system naturally gives rise to a $\mathbb{Z}_2$-symmetric figure eight loop. Crucially, this geometric structure is not an artificial construct but actually exists in practical models, as evidenced by prior studies (cf. \cite{CHB}). The primary objective of this work is to characterize the dynamical transitions occurring near such figure eight loops under generic $\mathbb{Z}_2$-symmetric perturbations within the Filippov framework. Our main contributions include two aspects: First, employing differential manifold theory, we reveal the intrinsic quantities of unfolding the bifurcation and demonstrate that the bifurcation set is a codimension-two submanifold of the set of all $\mathbb{Z}_2$-symmetric Filippov systems considered. Second, we derive an explicit non-degenerate condition with respect to parameters and, through a synthesis of displacement map method and asymptotic analysis, systematically establish the bifurcation diagram with precise asymptotic descriptions for all bifurcation boundaries.

This paper is organized as follows. In Section 2 we shortly review some basic notions on Filippov systems involved in this paper. In Section 3 we set up the problem of this paper and then state our main theorems. After introducing two preliminary lemmas in Section 4, we give the proofs of main theorems in Section 5 and Section 6. Finally, an example is showed in Section 7 to realize the bifurcation obtained in this paper.

%%%%%%%%%%%%%%%%%%%%%%%%%%%%%%%%
\section{Notions}

In this section we give a short review of discontinuous piecewise-smooth vector fields with emphasis on (\ref{mainsystem}), see \cite{YAK, MG1, MD, AFF} for more details. Throughout this paper, we call $\Sigma$ the {\it discontinuity boundary} of (\ref{mainsystem}) and denote by $\mathfrak{X}^2=\mathfrak{X}\times\mathfrak{X}$ the set of all vector fields $Z=(Z^+, Z^-)$ of form (\ref{mainsystem}), which can be endowed with the product topology.

Since $Z$ is discontinuous, i.e., $Z^+(x,y)\not\equiv Z^-(x,y)$ on $\Sigma$, the the definition of solutions of smooth vector fields is only suitable for the solutions that do not interact with $\Sigma$, while for the solutions of $Z$ that reach $\Sigma$ at some time, in this paper we use the Filippov's convention \cite{AFF} to define them. According to this convention, $\Sigma$ is separated into the {\it crossing set}
$$\Sigma^c:=\left\{(x, y)\in\Sigma: Z^+h(x,y)Z^-h(x,y)>0\right\}$$
and the {\it sliding set}
$$\Sigma^s:=\left\{(x, y)\in\Sigma: Z^+h(x,y)Z^-h(x,y)\le0\right\},$$
where $Z^\pm h=\langle Z^\pm,\nabla h\rangle$ and $\langle \cdot,\cdot\rangle$ denotes inner product. A sliding segment in the interior of $\Sigma^s$ is said to be {\it stable} if $Z^+h(x,y)<0<Z^-h(x,y)$ and {\it unstable} if $Z^+h(x,y)>0>Z^-h(x,y)$.

On $\Sigma^c$, both $Z^+$ and $Z^-$ are transverse to $\Sigma$ and their normal components have the same sign, which leads that
the solution reaching $\Sigma$ at a point in $\Sigma^c$ will cross $\Sigma$. On $\Sigma^s$, either both $Z^+$ and $Z^-$ are transverse to $\Sigma$ and their normal components have the opposite sign, or at least one of normal components is zero. In this case, by the Filippov convex method \cite{AFF} the solution reaching $\Sigma$ at a point in $\Sigma^s$ is allowed to slide along $\Sigma^s$, and the sliding dynamics obeys the vector field
$$
Z^s(x,y):=\mu Z^-(x,y)+(1-\mu)Z^+(x,y),\qquad (x,y)\in\Sigma^s,
$$
where $\mu$ is selected to ensure that $Z^s$ is tangent to $\Sigma$, i.e.,
$$\mu Z^-h(x,y)+(1-\mu)Z^+h(x,y)=0,\qquad (x,y)\in\Sigma^s.$$
Clearly,
$\mu={Z^+h(x,y)}/({Z^+h(x,y)-Z^-h(x,y)})$
for $Z^-h(x,y)-Z^+h(x,y)\ne0$. Here $Z^s$ is called the {\it sliding vector field} and its an equilibrium is called a {\it pseudo-equilibrium} of $Z$ \cite{YAK}. In brief, the solutions of $Z$ that do interact with $\Sigma$ can be constructed by concatenating the standard solutions in $\Sigma^\pm$ and the sliding solutions in $\Sigma$, see \cite{YAK, AFF} for more details.

The set
$$\partial\Sigma^s:=\left\{(x, y)\in\Sigma: Z^+h(x,y)Z^-h(x,y)=0\right\}$$
plays an important role in the bifurcation analysis of Filippov systems. Following \cite{MG1,YAK}, we recall some notions related to $\partial\Sigma^s$. Let $p\in\partial\Sigma^s$. Then it is a {\it boundary equilibrium} of $Z^\pm$ if $Z^\pm(p)=0$, or a {\it tangent point} of $Z^\pm$ if $Z^\pm(p)\ne0$ and $Z^\pm h(p)=0$, or a {\it regular point} of $Z^\pm$ if $Z^\pm h(p)\ne0$. In addition, a tangent point $p$ of $Z^\pm$ is said to be a {\it fold} if $(Z^\pm)^2h(p)\ne0$.
$p$ is called a {\it regular-fold} of $Z$ if it is a fold of one sub-vector field and a regular point of the other, and a {\it fold-fold} of $Z$ if it is a fold of both sub-vector fields. A fold $p$ of $Z^+$ is said to be {\it visible} (resp. {\it invisible}) if $(Z^+)^2h(p)>0$ (resp. $<0$), and a fold $p$ of $Z^-$ is said to be {\it visible} (resp. {\it invisible}) if $(Z^-)^2h(p)<0$ (resp. $>0$).

There are some different criteria for distinguishing and naming the periodic orbits and homoclinic orbits of $Z\in\mathfrak{X}^2$ (cf. \cite{YAK, MG1, AGN}). Therefore, before formally introducing our work, it is necessary to clarify the criterion adopted in this paper in order not to cause confusion. First, a periodic orbit that lies totally in $\Sigma^+$ or $\Sigma^-$ is called a {\it standard periodic orbit},
and a closed curve formed by concatenating the regular orbits of two sub-vector fields only at some
points of $\Sigma^c$ is called a {\it crossing periodic orbit}. Besides, $Z$ can have a closed curve that consists of regular orbits and tangent points.
In this case, we treat the closed curve as a periodic orbit rather than a homoclinic orbit in the sense that the travelling time to a tangent point is finite, and we call it a {\it tangential periodic orbit}. In particular, a tangential periodic orbit can be classified into the following three cases.
\vspace{-12pt}
\begin{itemize}
\setlength{\itemsep}{0mm}
\item[(i)] {\it Sliding periodic orbit}, which contains a sliding segment in $\Sigma^s\setminus\partial\Sigma^s$. A sliding periodic orbit is said to be {\it stable} (resp. {\it unstable}) if its sliding segment is stable (resp. unstable).
\item[(ii)] {\it Critical crossing periodic orbit}, which occupies $\Sigma^+$ and $\Sigma^-$ and intersects $\Sigma$ only at some points in the closure of $\Sigma^c$.
    \item[(iii)] {\it Grazing periodic orbit}, which lies totally in $\Sigma^+\cup\Sigma$ or $\Sigma^-\cup\Sigma$ and intersects $\Sigma$ only at tangent points. Clearly,
 a grazing periodic orbit must be a periodic orbit of a sub-vector field.
\end{itemize}
\vspace{-12pt}
An isolated standard (resp. crossing, sliding, critical crossing, and grazing) periodic orbit of $Z\in\mathfrak{X}^2$ in the set of all periodic orbits is called a {\it standard} (resp. {\it crossing, sliding, critical crossing}, and {\it grazing}) cycle.

Finally, a closed curve of $Z\in\mathfrak{X}^2$ that consists of regular orbits and a unique equilibrium, including standard equilibrium in $\Sigma^+$ and $\Sigma^-$, pseudo-equilibrium in $\Sigma^s\setminus\partial\Sigma^s$ and boundary equilibrium in $\partial\Sigma^s$, is called a {\it homoclinic orbit}. In particular, if a homoclinic orbit contains a sliding segment, it is called a {\it sliding homoclinic orbit}.

%%%%%%%%%%%%%%%%%%%%%%%%%%%%%%%%%%%%%%%%%%%%%%%%%%%%%%%%%%%%
\section{Main results}
\setcounter{equation}{0}
\setcounter{lm}{0}
\setcounter{thm}{0}
\setcounter{rmk}{0}
\setcounter{df}{0}
\setcounter{cor}{0}

This section is devoted to setting  up our problem and stating the main results. From now on, for $Z=(Z^+, Z^-)\in\mathfrak{X}^2$, we assume $k\ge2$ and always take $h(x,y)=y$, i.e., $\Sigma$ is the $x$-axis, because only the discontinuity boundary $\Sigma$ in the vicinity of a fold-fold is involved in this paper as we will see.

Let $\Omega_0$ be the set formed by the vector field $Z_0=(Z^+_0,Z^-_0)\in\mathfrak{X}^2$ satisfying the following assumptions.
\vspace{-14pt}
\begin{description}
\setlength{\itemsep}{0mm}
\item[(H1)] $Z^+_0$ has an anticlockwise rotary $T_0$-periodic hyperbolic limit cycle $\Gamma_0$ lying in $\overline{\Sigma^+}=\{(x,y): x^2+y^2<r^2, y\ge0\}$ and touching $\Sigma$ at a unique point, which lies at the origin $O$ and is a fold of $Z^+_0$, where $T_0$ is the minimal positive period. Thus
    \begin{equation}\label{cewr343}
    f^+_0(0,0)>0,\qquad g^+_0(0,0)=0,\qquad g^+_{0x}(0,0)>0,
    \end{equation}
    where $(f^+_0, g^+_0)$ is the coordinates of $Z^+_0$ and the subscript $x$ denotes partial derivative.
\setlength{\itemsep}{0mm}
\item[(H2)] $Z_0$ is $\mathbb{Z}_2$-symmetric with respect to $O$, i.e., $Z_0^-(x,y)=-Z_0^+(-x,-y)$ for all $(x,y)\in N$.
\end{description}
Then each $Z_0\in\Omega_0$ has a figure eight loop $\Upsilon_0$ kinking at the fold-fold $O$. Our goal is to explore what typically happens in a sufficiently small figure eight annulus neighborhood of $\Upsilon_0$ when $Z_0$ is perturbed in $\Omega$, where $\Omega\subset\mathfrak{X}^2$ is the set of all $\mathbb{Z}_2$-symmetric vector fields with respect to $O$.

First we clarify the codimension of the set of unperturbed vector fields in the perturbation class.

\begin{thm}\label{thm-codim}
Assume that $Z_0\in\Omega_0$ has a figure eight loop $\Upsilon_0$ characterized by {\bf(H1)} and {\bf (H2)}. Then for a sufficiently small figure eight annulus neighborhood $\mathcal{A}$ of $\Upsilon_0$ there exists a neighborhood $\mathcal{U}\subset\Omega$ of $Z_0$ such that $\mathcal{U}_0$ is a codimension-2 $C^k$ submanifold of $\mathcal{U}$, where $\mathcal{U}_0$ is the set of all vector fields of $\mathcal{U}$ having a figure eight loop characterized by {\bf(H1)} and {\bf (H2)} in $\mathcal{A}$.
\end{thm}

Theorem~\ref{thm-codim} is proved in Section 5.

Next we consider the following two-parametric perturbations of $Z_0\in\Omega_0$ to study the bifurcation phenomena in a sufficiently small figure eight annulus neighborhood $\mathcal{A}$ of $\Upsilon_0$,
\begin{eqnarray}
Z(x, y;\alpha)=\left\{
\begin{aligned}
  &Z^+(x, y;\alpha)=(f^+(x,y;\alpha),g^+(x,y;\alpha))
~~~~&& {\rm if}~(x, y)\in\Sigma^+,\\
  &Z^-(x, y;\alpha)=(-f^+(-x,-y;\alpha),-g^+(-x,-y;\alpha))
~~~~&& {\rm if}~(x, y)\in\Sigma^-,\\
\end{aligned}
\right.
\label{parasys}
\end{eqnarray}
which $C^k$ smoothly depends on $\alpha=(\alpha_1,\alpha_2)\in U\subset\mathbb{R}^2$, where $Z(x,y;0)=Z_0(x,y)$ and $U$ is a neighborhood of $\alpha=0$ such that $Z(x, y;\alpha)\in\mathcal{U}$ for $\alpha\in U$. To state the bifurcation result, we introduce
\begin{equation}\label{ceirucew434}
\begin{aligned}
\kappa_i&:=\int^{T_0}_0\lambda(t)(f^+g^+_{\alpha_i}-g^+f^+_{\alpha_i})(\gamma_0(t);0)dt,\quad i=1,2,\\
\lambda(t)&:=\exp\left(\int^{T_0}_t(f^+_x+g^+_y)(\gamma_0(s);0)ds\right),\qquad t\in[0,T_0],
\end{aligned}
\end{equation}
where $\gamma_0(t)$ is the solution associated to $\Gamma_0$ satisfying $\gamma_0(0)=O$, and the subscripts $x,y,\alpha_i$ are the corresponding partial derivatives.

As we have known from the qualitative theory of smooth dynamical systems, $\Gamma_0$ is a hyperbolic limit cycle of $Z_0^+$ if and only if $\lambda(0)\ne1$, see e.g., \cite{LP}. In particular, $\lambda(0)<1$ (resp. $>1$) means that $\Gamma_0$ is stable (resp. unstable). In the following theorem we only state the bifurcation result for $\lambda(0)<1$ because the case of $\lambda(0)>1$ can be obtained directly from the stated result by the transformation $(x,y,t)\rightarrow(-x,y,-t)$.

\begin{thm}\label{mainthm}
Assume that $Z_0\in\Omega_0$ has a figure eight loop $\Upsilon_0$ characterized by {\bf(H1)}, {\bf (H2)} and $\lambda(0)<1$, and consider its two-parametric perturbation $Z(x,y;\alpha)$ for $\alpha\in U$ given in $(\ref{parasys})$. If
\begin{equation}\label{transversality}
\kappa_1g^+_{\alpha_2}(0,0;0)-\kappa_2g^+_{\alpha_1}(0,0;0)\ne0,
\end{equation}
then for a sufficiently small figure eight annulus neighborhood $\mathcal{A}$ of $\Upsilon_0$ there exists a neighborhood $U^*\subset U$ of $\alpha=0$, a locally smooth invertible reparameterization $(\beta_1,\beta_2)=(\varphi_1(\alpha),\varphi_2(\alpha))$ with $\varphi_1(0)=\varphi_2(0)=0$ for $\alpha\in U^*$, and smooth functions $\psi_i(\beta_1)$ $(i=1,2,3,4,5)$ defined in $\varphi_1(U^*)$, which are quadratically tangent to $\beta_2=0$ at $(\beta_1,\beta_2)=(0,0)$ and satisfy
\begin{equation}\label{cewr324fwe4}
\begin{aligned}
&\psi_1(\beta_1)>\psi_2(\beta_1)>\psi_4(\beta_1)>0\qquad  &&for\quad \beta_1<0,\\
&\psi_3(\beta_1)<\psi_5(\beta_1)<0\qquad  &&for\quad \beta_1>0,
\end{aligned}
\end{equation}
such that the following statements hold in $\mathcal{A}$. These statements describe the bifurcation diagrams in Figures~\ref{symGSB} and \ref{symGSBR1}.
\vspace{-12pt}
\begin{itemize}
\setlength{\itemsep}{0mm}
\item[{\rm(1)}] For $\beta_1=0$ and $\beta_2\ne0$, the fold-fold $O$ persists, and the figure eight loop $\Upsilon_0$ becomes
    \begin{itemize}
    \item[{\rm(1a)}] two hyperbolically stable standard cycles for $\beta_2>0$;
    \item[{\rm(1b)}] a hyperbolically stable crossing cycle for $\beta_2<0$.
    \end{itemize}
\item[{\rm(2)}] For $\beta_1>0$, the fold-fold $O$ becomes two regular-folds $(\beta_1,0)$ and $(-\beta_1,0)$, giving rise to a stable sliding segment $\{(x,0):-\beta_1<x<\beta_1\}$ with a pseudo-saddle $O$, and the figure eight loop $\Upsilon_0$  becomes
    \begin{itemize}
    \item[{\rm(2a)}] two hyperbolically stable standard cycles for $\beta_2>0$;
    \item[{\rm(2b)}] two internally stable grazing cycles for $\beta_2=0$;
    \item[{\rm(2c)}] two stable one-zonal sliding cycles for $\psi_5(\beta_1)<\beta_2<0$;
    \item[{\rm(2d)}] two one-zonal sliding homoclinic orbits to the pseudo-saddle $O$ for $\beta_2=\psi_5(\beta_1)$;
    \item[{\rm(2e)}] a stable two-zonal sliding cycle for $\psi_3(\beta_1)<\beta_2<\psi_5(\beta_1)$;
    \item[{\rm(2f)}] an externally stable critical crossing cycle for $\beta_2=\psi_3(\beta_1)$;
    \item[{\rm(2g)}] a hyperbolically stable crossing cycle for $\beta_2<\psi_3(\beta_1)$.
    \end{itemize}
\item[{\rm(3)}] For $\beta_1<0$, the fold-fold $O$ becomes two regular-folds $(\beta_1,0)$ and $(-\beta_1,0)$, giving rise to an unstable sliding segment $\{(x,0):\beta_1<x<-\beta_1\}$ with a pseudo-saddle $O$, and the figure eight loop $\Upsilon_0$  becomes
    \begin{itemize}
    \item[{\rm(3a)}] two hyperbolically stable standard cycles for $\beta_2>\psi_1(\beta_1)$;
    \item[{\rm(3b)}] two hyperbolically stable standard cycles and an externally stable crossing cycle of multiplicity two for $\beta_2=\psi_1(\beta_1)$;
    \item[{\rm(3c)}] two hyperbolically stable standard cycles and two hyperbolic crossing cycles, where the inner crossing cycle is unstable and the outer one is stable, for $\psi_2(\beta_1)<\beta_2<\psi_1(\beta_1)$;
    \item[{\rm(3d)}] two hyperbolically stable standard cycles, an externally unstable critical crossing cycle, and a hyperbolically stable crossing cycle for $\beta_2=\psi_2(\beta_1)$;
    \item[{\rm(3e)}] two hyperbolically stable standard cycles, an unstable two-zonal sliding cycle, and a hyperbolically stable crossing cycle for $\psi_4(\beta_1)<\beta_2<\psi_2(\beta_1)$;
    \item[{\rm(3f)}] two hyperbolically stable standard cycles, two one-zonal sliding homoclinic orbits to the pseudo-saddle $O$, and a hyperbolically stable crossing cycle for $\beta_2=\psi_4(\beta_1)$;
    \item[{\rm(3g)}] two hyperbolically stable standard cycles, two unstable one-zonal sliding cycles, and a hyperbolically stable crossing cycle for $0<\beta_2<\psi_4(\beta_1)$;
    \item[{\rm(3h)}] two internally stable grazing cycles and a hyperbolically stable crossing cycle for $\beta_2=0$;
    \item[{\rm(3i)}] a hyperbolically stable crossing cycle for $\beta_2<0$.
    \end{itemize}
\end{itemize}
\end{thm}

\begin{figure}[h]
  \begin{minipage}[t]{1.0\linewidth}
  \centering
  \includegraphics[width=5.2in]{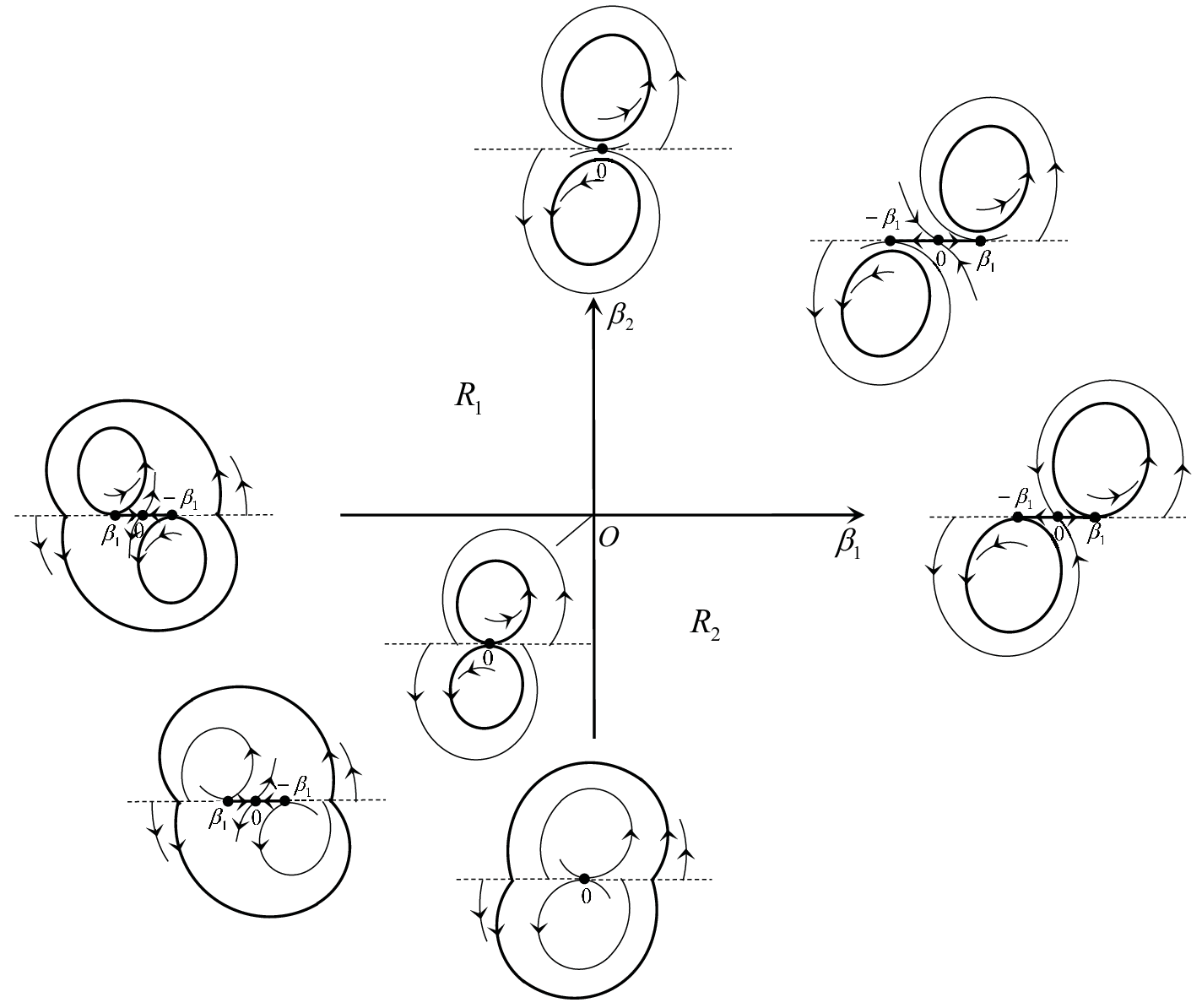}
  \end{minipage}
\caption{{\small Grazing-sliding bifurcation in $\mathbb{Z}_2$-symmetric Filippov systems except regions $R_1, R_2$.}}
\label{symGSB}
\end{figure}

\begin{figure}
\centering{}
\subfigure[{\small Bifurcations in the region $R_1$}]{
	\scalebox{0.55}[0.55]{
	\includegraphics{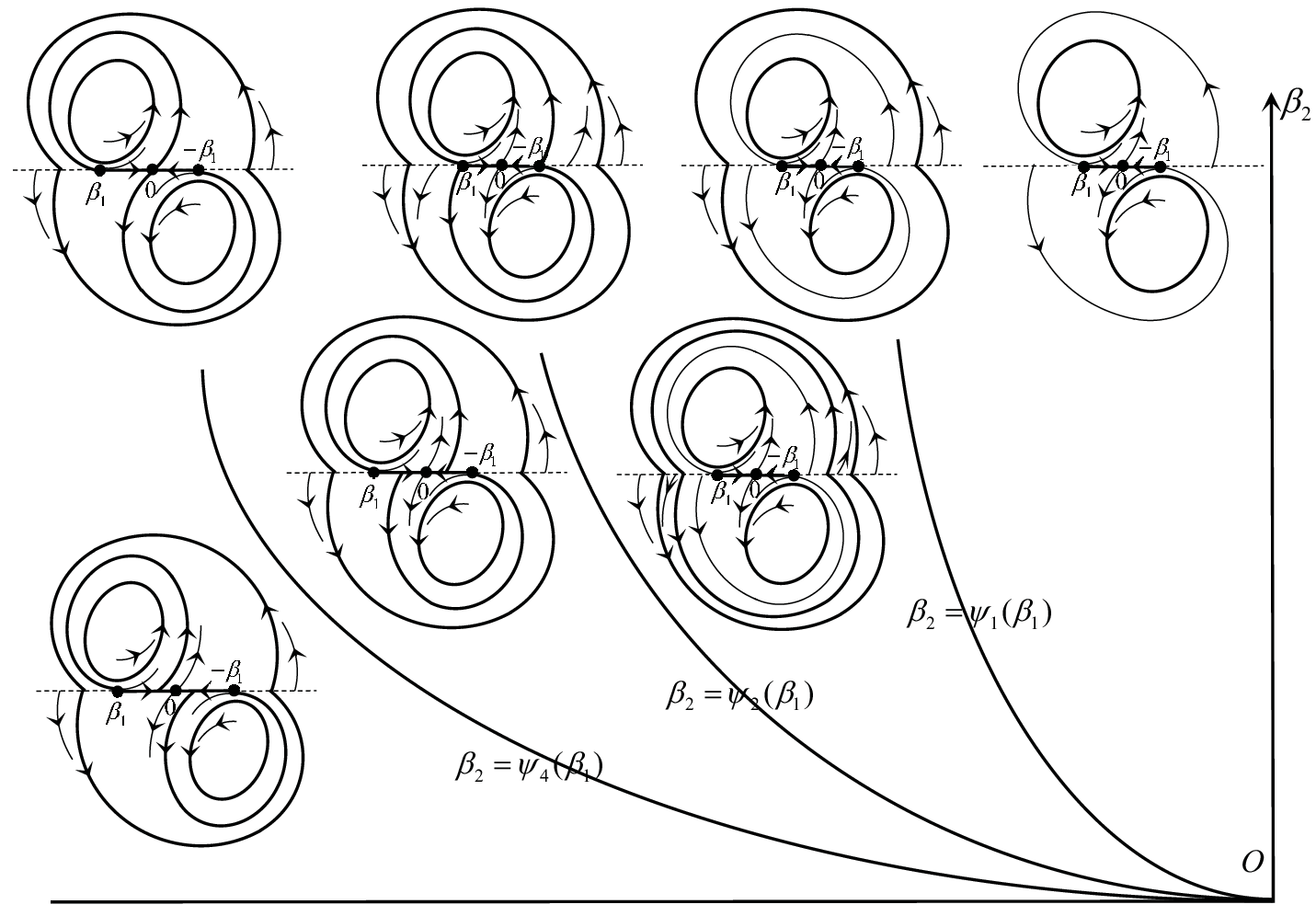}
	}
}
\subfigure[{\small Bifurcations in the region $R_2$}]{
	\scalebox{0.55}[0.55]{
	\includegraphics{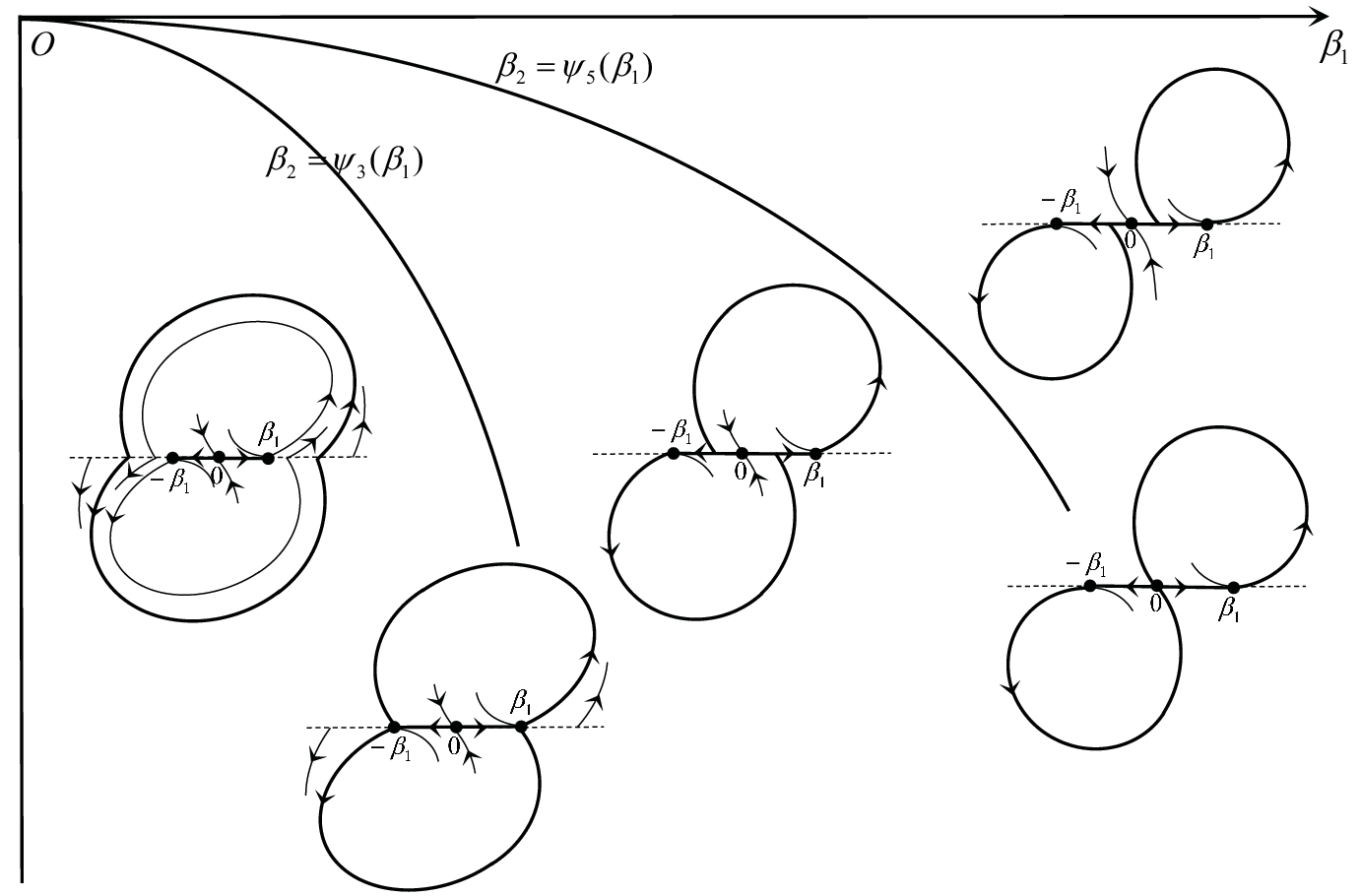}
	}
}
\caption{\footnotesize {\small Bifurcations in the regions $R_1$ and $R_2$}. }
\label{symGSBR1}
\end{figure}

%%%%%%%%%%%%%%%%%%%%%%%%%%%%%%%%%%%%%%%%%%%%%%%%%%%%%%%%%%%%
\section{Preliminary lemmas}
\setcounter{equation}{0}
\setcounter{lm}{0}
\setcounter{thm}{0}
\setcounter{rmk}{0}
\setcounter{df}{0}
\setcounter{cor}{0}

This section establishes two preliminary lemmas that underpin the proofs of main theorems.

\begin{lm}\label{slidynamics}
Under the assumption of {\bf(H1)}, for a sufficiently small annulus neighborhood $\mathcal{A}^+$ of $\Gamma_0$ there exists a neighborhood $\mathcal{U}_1^+\subset\mathfrak{X}$ of $Z_0^+$ and a $C^k$ map $\varphi_1(Z^+)$ with $\varphi_1(Z^+_0)=0$ defined in $\mathcal{U}_1^+$ such that each $Z^+\in\mathcal{U}^+_1$ has a unique tangent point in $\mathcal{A}^+$, which lies at $(\varphi_1(Z^+),0)$ and is a fold satisfying
\begin{equation}\label{sdkg}
g^+(\varphi_1(Z^+),0)=0,\qquad f^+(\varphi_1(Z^+),0)>0,\qquad g^+_x(\varphi_1(Z^+),0)>0,
\end{equation}
where $(f^+,g^+)$ is the coordinates of $Z^+$.
\end{lm}

\begin{proof}
Consider the Fr\'echet differentiable map
$$\mathcal{F}(Z^+,p):=g^+(p,0)$$
for $Z^+\in\mathfrak{X}$ and $p\in\mathbb{R}$. From (\ref{cewr343}),
$$
\mathcal{F}(Z^+_0,0)=g^+_0(0,0)=0,\qquad \frac{\partial\mathcal{F}(Z^+_0,0)}{\partial p}=g^+_{0x}(0,0)>0.
$$
Thus, by the Implicit Function Theorem \cite{KCCM}, there exist a neighborhood $\mathcal{U}_1^+$ of $Z_0^+$ and a unique and $C^k$ map $p=\varphi_1(Z^+)$ defined in $ \mathcal{U}^+_1$ such that
$$
\varphi_1(Z_0^+)=0,\qquad \mathcal{F}(Z^+,\varphi_1(Z^+))=g^+(\varphi_1(Z^+),0)=0,
$$
i.e., $(\varphi_1(Z^+),0)$ is the unique tangent point of $Z^+\in\mathcal{U}^+_1$ in $\mathcal{A}^+$.
In addition, due to $f^+_0(0,0)>0$ and $g^+_{0x}(0,0)>0$ in (\ref{cewr343}), $\mathcal{U}_1^+$ can be chosen to ensure that
$$
f^+(\varphi_1(Z^+),0)>0,\quad \quad g^+_x(\varphi_1(Z^+),0)>0.
$$
This implies that $(\varphi_1(Z^+),0)$ is a fold of $Z^+\in\mathcal{U}^+_1$ satisfying (\ref{sdkg}). Hence, the proof is completed.
\end{proof}

Note that the limit cycle $\Gamma_0$ of $Z_0^+$ is assumed to be hyperbolic in {\bf(H1)}. This means that each vector field in a small neighborhood of $Z_0^+$ always has a hyperbolic limit cycle preserving the stability by the bifurcation theory of smooth dynamical systems. The following lemma characterizes the positional relationship of the limit cycle and $\Sigma$.

\begin{lm}\label{csin34f}
Let $\mathcal{U}_1^+$ be given in Lemma~\ref{slidynamics}. Under the assumption of {\bf(H1)}, for a sufficiently small annulus neighborhood $\mathcal{A}^+$ of $\Gamma_0$ there exists a neighborhood $\mathcal{U}_2^+\subset\mathcal{U}_1^+$ of $Z_0^+$ and a $C^k$ map $\varphi_2(Z^+)$ defined in $\mathcal{U}_2^+$ such that each $Z^+\in\mathcal{U}_2^+$ has a unique limit cycle $\Gamma_{Z^+}$ in $\mathcal{A}^+$, which passes through $(\varphi_1(Z^+),\varphi_2(Z^+))$ and has the same hyperbolicity and stability as $\Gamma_0$. Moreover, $\Gamma_{Z^+}$ has exactly zero $($ resp. one, two$)$ intersections with $\Sigma$ if $\varphi_2(Z^+)>0$ $($resp. $=0, <0$$)$.
\end{lm}

\begin{proof}

By the change
$$
\mathcal{T}: (x,y)\rightarrow(x+\varphi_1(Z^+),y),
$$
$Z^+\in\mathcal{U}_1^+$ is transformed to a new vector field, denoted by $\widehat Z^+:=\mathcal{T}(Z^+)$. In particular, $\widehat Z^+_0:=\mathcal{T}(Z^+_0)=Z^+_0$ due to $\varphi_1(Z^+_0)=0$, as obtained in Lemma~\ref{slidynamics}. In this case, the fold of $Z^+$, namely $(\varphi_1(Z^+),0)$, is translated to $O$, i.e., $O$ is always a fold of $\widehat Z^+\in\mathcal{T}(\mathcal{U}_1^+)$.

Under the assumption of {\bf(H1)}, for a sufficiently small annulus neighborhood $\mathcal{A}^+$ of $\Gamma_0$ there is a neighborhood $\widehat{\mathcal{U}}^+_2\subset \mathcal{T}(\mathcal{U}_1^+)$ of $\widehat Z_0^+$ such that for $\widehat Z^+\in\widehat{\mathcal{U}}^+_2$ we can define a Poincar\'e map $P(y;\widehat Z^+)$ having the $y$-axis near $O$ as the Poincar\'e section by the forward orbits of $\widehat Z^+$. Clearly, $P(0;\widehat Z_0^+)=0$. Moreover, it follows from \cite{LP} and the hyperbolicity of $\Gamma_0$ that $\partial P(0;\widehat Z^+_0)/\partial y\ne1$. Therefore, a direct application of the Implicit Function Theorem yields that $\widehat{\mathcal{U}}_2^+$ can be reduced to get a unique and $C^k$ map $\widehat\varphi_2(\widehat Z^+)$ defined in $\widehat{\mathcal{U}}_2^+$ such that $\widehat\varphi_2(\widehat Z^+_0)=0$, $P(\widehat\varphi_2(\widehat Z^+),\widehat Z^+)=\widehat\varphi_2(\widehat Z^+)$ and $(\partial P(\widehat\varphi_2(\widehat Z^+);\widehat Z^+)/\partial y-1)(\partial P(0;\widehat Z^+_0)/\partial y-1)>0$. This means that the orbit of $\widehat Z^+$ passing through $(0,\widehat\varphi_2(\widehat Z^+))$ is a limit cycle, which is hyperbolic and has the same stability as $\Gamma_0$. It follows from the uniqueness of $\widehat\varphi_2(\widehat Z^+)$ that the limit cycle of $\widehat Z^+$ in $\mathcal{A}^+$ is unique. Moreover, since $O$ is always a fold of $\widehat Z^+$, it is not difficult to see that the limit cycle has exactly zero $($ resp. one, two$)$ intersections with $\Sigma$ if $\widehat\varphi_2(\widehat Z^+)>0$ $($resp. $=0, <0$$)$.

Finally, letting $\mathcal{U}_2^+=\mathcal{T}^{-1}(\widehat{\mathcal{U}}^+_2)$ and $\varphi_2(Z^+)=\widehat\varphi_2(\mathcal{T}(Z^+))$,
we complete the proof of lemma.
\end{proof}

%%%%%%%%%%%%%%%%%%%%%%%%%%%%%%%%%%%%%%%%%%%%%%%%%%%%%%%%%%%%
\section{Proof of Theorem~\ref{thm-codim}}
\setcounter{equation}{0}
\setcounter{lm}{0}
\setcounter{thm}{0}
\setcounter{rmk}{0}
\setcounter{df}{0}
\setcounter{cor}{0}

To prove Theorem~\ref{thm-codim}, we need the following lemma.

\begin{lm}\label{co123}
Let $\varphi_1(Z^+)$ $($resp. $\varphi_2(Z^+)$$)$ be the map defined in $\mathcal{U}_1^+$ $($resp. $\mathcal{U}_2^+\subset\mathcal{U}_1^+$$)$ and obtained in Lemma~\ref{slidynamics} $($resp. Lemma~\ref{csin34f}$)$. Given $Z^+\in\mathcal{U}_2^+$, for any $(c_1,c_2)\in\mathbb{R}^2$ there is a constant $\varepsilon_0>0$ and a smooth curve $\ell(\varepsilon):(-\varepsilon_0,\varepsilon_0)\rightarrow\mathcal{U}^+_2$ such that $\ell(0)=Z^+$ and
$$(\varphi_1(\ell(\varepsilon)),\varphi_2(\ell(\varepsilon)))=(\varphi_1(Z^+), \varphi_2(Z^+))+(c_1,c_2)\varepsilon+\mathcal{O}(\varepsilon^2).$$
\end{lm}

\begin{proof}
Given $Z^+=(f^+,g^+)\in\mathcal{U}^+_2$, for any $(c_1,c_2)\in\mathbb{R}^2$ we consider $Z^+_\varepsilon=(X(x,y;\varepsilon), Y(x,y;\varepsilon))$ with
\begin{equation}\label{53fsfvfD}
X(x,y;\varepsilon):=f^+(x,y)+\varepsilon cg^+(x,y),\qquad Y(x,y;\varepsilon):=g^+(x,y)-\varepsilon c_1g^+_x(\varphi_1(Z^+),0),
\end{equation}
where
\begin{equation}\label{oiajfca343}
\begin{aligned}
c:=&c_1\frac{\int^{T_{Z^+}}_0\lambda_{Z^+}(t)(f^+g^+_x-g^+f^+_x)(\gamma_{Z^+}(t))dt-
g^+_x(\varphi_1(Z^+),0)\int^{T_{Z^+}}_0\lambda_{Z^+}(t)f^+(\gamma_{Z^+}(t))dt}{\int^{T_{Z^+}}_0\lambda_{Z^+}(t)(g^+)^2(\gamma_{Z^+}(t))dt}\\
&+c_2\frac{\left(\lambda_{Z^+}(0)-1\right)f^+(\varphi_1(Z^+), \varphi_2(Z^+))}{\int^{T_{Z^+}}_0\lambda_{Z^+}(t)(g^+)^2(\gamma_{Z^+}(t))dt},\\
\lambda_{Z^+}(t):=&\exp\int_t^{T_{Z^+}}(f^+_x+g^+_y)(\gamma_{Z^+}(s))ds,
\end{aligned}
\end{equation}
$\gamma_{Z^+}(t)$ is the coordinates of the limit cycle $\Gamma_{Z^+}$ of $Z^+$ with $\gamma_{Z^+}(0)=(\varphi_1(Z^+),\varphi_2(Z^+))$ and $T_{Z^+}$ is the minimal positive period of $\Gamma_{Z^+}$. Here the existence of $\Gamma_{Z^+}$ is obtained in Lemma~\ref{csin34f}.
Clearly, there is a constant $\varepsilon_1>0$ such that $Z^+_\varepsilon\in\mathcal{U}^+_2$ for $\varepsilon\in(-\varepsilon_1,\varepsilon_1)$. This allows us to compute $\varphi_1(Z^+_\varepsilon)$ and $\varphi_2(Z^+_\varepsilon)$ as follows.

From (\ref{sdkg}),
$$
\begin{aligned}
&X(\varphi_1(Z^+),0;0)=f^+(\varphi_1(Z^+),0)>0,\qquad Y(\varphi_1(Z^+),0;0)=g^+(\varphi_1(Z^+),0)=0,\\
&Y_x(\varphi_1(Z^+),0;0)=g^+_x(\varphi_1(Z^+),0)>0.
\end{aligned}
$$
By the Implicit Function Theorem and the sign-preserving property of continuous functions, there is a constant $\varepsilon_2\in(0,\varepsilon_1)$ and a smooth map $\phi_1(\varepsilon)$ defined in $(-\varepsilon_2,\varepsilon_2)$ such that $\phi_1(0)=\varphi_1(Z^+)$ and
$$X(\phi_1(\varepsilon),0;\varepsilon)>0,\qquad Y(\phi_1(\varepsilon),0;\varepsilon)=0,\qquad Y_x(\phi_1(\varepsilon),0;\varepsilon)>0.$$
This together with Lemma~\ref{slidynamics} means that
$$\varphi_1(Z^+_\varepsilon)=\phi_1(\varepsilon),\qquad \varepsilon\in(-\varepsilon_2,\varepsilon_2).$$
Therefore, due to $\phi'_1(0)=c_1$,
we can write $\varphi_1(Z^+_\varepsilon)$ in the form
\begin{equation}\label{cgsks43}
\varphi_1(Z^+_\varepsilon)=\varphi_1(Z^+)+c_1\varepsilon+\mathcal{O}(\varepsilon^2).
\end{equation}

On the other hand, by the linear change
\begin{equation}\label{nfe34fe}
(x,y)\rightarrow(x+\varphi_1(Z^+_\varepsilon)-\varphi_1(Z^+),y)
\end{equation}
we transform $Z^+_\varepsilon=(X(x,y;\varepsilon), Y(x,y;\varepsilon))$ to $\widehat Z^+_\varepsilon=(\widehat X(x,y;\varepsilon),\widehat Y(x,y;\varepsilon))$ with
\begin{equation}\label{ckd34fsr}
\widehat X(x,y;\varepsilon):=X(x+\varphi_1(Z^+_\varepsilon)-\varphi_1(Z^+),y;\varepsilon),\qquad \widehat Y(x,y;\varepsilon):=Y(x+\varphi_1(Z^+_\varepsilon)-\varphi_1(Z^+),y;\varepsilon).
\end{equation}
Then for $\widehat Z^+_\varepsilon$ we can define a Poincar\'e map $P(y;\varepsilon)$ around the limit cycle $\Gamma_{Z^+}$ by choosing $x=\varphi_1(Z^+)$ near $(\varphi_1(Z^+), \varphi_2(Z^+))$ as the Poincar\'e section. Since
$$P(\varphi_2(Z^+);0)=\varphi_2(Z^+),\qquad \frac{\partial P(\varphi_2(Z^+);0)}{\partial y}\ne1,$$
there is a constant $\varepsilon_3\in(0,\varepsilon_2)$ and a smooth map $\phi_2(\varepsilon)$ defined in $(-\varepsilon_3,\varepsilon_3)$ such that $\phi_2(0)=\varphi_2(Z^+)$ and $P(\phi_2(\varepsilon),\varepsilon)=\phi_2(\varepsilon)$. This means that the orbit of $\widehat Z^+_\varepsilon$ passing through $(\varphi_1(Z^+),\phi_2(\varepsilon))$ is a limit cycle. Thus the orbit of $Z^+_\varepsilon$ passing through $(\varphi_1(Z^+_\varepsilon),\phi_2(\varepsilon))$ is a limit cycle by (\ref{nfe34fe}). Besides, it follows from  $Z^+_\varepsilon\in\mathcal{U}^+_2$ and Lemma~\ref{csin34f} that $Z^+_\varepsilon$ has a unique limit cycle in $\mathcal{A}^+$, which passes through $(\varphi_1(Z^+_\varepsilon),\varphi_2(Z^+_\varepsilon))$. Therefore,
\begin{equation}\label{nfe34feafaf}
\varphi_2(Z^+_\varepsilon)=\phi_2(\varepsilon),\qquad \varepsilon\in(-\varepsilon_3,\varepsilon_3).
\end{equation}

Furthermore, it follows from \cite[p.384]{AAA}, (\ref{53fsfvfD}), (\ref{oiajfca343}) and (\ref{ckd34fsr}) that
\begin{equation}\label{oiajf3}
\frac{\partial P(\varphi_2(Z^+);0)}{\partial y}=\exp\int_0^{T_{Z^+}}(\widehat X_x+\widehat Y_y)(\gamma_{Z^+}(s);0)ds=\lambda_{Z^+}(0)
\end{equation}
and
\begin{equation}\label{oiajafaf3}
\begin{aligned}
\frac{\partial P(\varphi_2(Z^+);0)}{\partial\varepsilon}=&~\frac{1}{\widehat X(\varphi_1(Z^+), \varphi_2(Z^+);0)}\int^{T_{Z^+}}_0\!\!\!\Bigg\{\!\exp\!\int_t^{T_{Z^+}}\!\!(\widehat X_x+\widehat Y_y)(\gamma_{Z^+}(s);0)ds\left(\widehat X\widehat Y_\varepsilon-\widehat Y\widehat X_\varepsilon\right)(\gamma_{Z^+}(t);0)\!\Bigg\}dt\\
=&~\frac{1}{f^+(\varphi_1(Z^+), \varphi_2(Z^+))}\Bigg\{c_1\int^{T_{Z^+}}_0\lambda_{Z^+}(t)(f^+g^+_x-g^+f^+_x)(\gamma_{Z^+}(t))dt+\\
&\int^{T_{Z^+}}_0\lambda_{Z^+}(t)\left(-g^+_x(\varphi_1(Z^+),0)c_1f^+(\gamma_{Z^+}(t))-c(g^+)^2(\gamma_{Z^+}(t))\right)dt\Bigg\}.
\end{aligned}
\end{equation}
Using (\ref{oiajfca343}), (\ref{oiajf3}) and (\ref{oiajafaf3}), we obtain
\begin{equation}\label{oiaafaafajf3}
\begin{aligned}
\phi_2'(0)&=-\frac{\partial P(\varphi_2(Z^+);0)}{\partial\varepsilon}/\left(\frac{\partial P(\varphi_2(Z^+);0)}{\partial y}-1\right)=c_2.
\end{aligned}
\end{equation}
Hence, by (\ref{nfe34feafaf}) and (\ref{oiaafaafajf3}), $\varphi_2(Z^+_\varepsilon)$ can be written in the form
\begin{equation}\label{cgskcafawr}
\varphi_2(Z^+_\varepsilon)=\varphi_2(Z^+)+c_2\varepsilon+\mathcal{O}(\varepsilon^2).
\end{equation}
Finally, taking $\varepsilon_0=\varepsilon_3$ and $\ell(\varepsilon)=Z^+_\varepsilon$, we complete the proof from (\ref{cgsks43}) and (\ref{cgskcafawr}).
\end{proof}

\begin{proof}[{\bf Proof of Theorem~\ref{thm-codim}}]
For a sufficiently small figure eight annulus neighborhood $\mathcal{A}$ of $\Upsilon_0$ we take $\mathcal{U}=\mathcal{U}_2^+\times\mathcal{U}_2^-$, where $\mathcal{U}_2^+$ is the neighborhood of $Z_0^+$ given in Lemma~\ref{csin34f} and $\mathcal{U}_2^-$ is the neighborhood of $Z_0^-$ such that the $\mathbb{Z}_2$-symmetric counterpart of $Z^-\in\mathcal{U}_2^-$ lies in $\mathcal{U}_2^+$. Then $Z\in\mathcal{U}$ has a figure eight loop characterized by {\bf(H1)} and {\bf (H2)} in $\mathcal{A}$ if and only if $\Lambda(Z)=(0,0)$ by Lemma~\ref{slidynamics} and \ref{csin34f}, where $\Lambda(Z):\mathcal{U}\rightarrow\mathbb{R}^2$ is defined as $\Lambda(Z)=(\varphi_1(Z^+),\varphi_2(Z^+))$. Clearly, $\Lambda$ is $C^k$ because both $\varphi_1(Z^+)$ and $\varphi_2(Z^+)$ are $C^k$.

On the other hand, we show that the derivative $D\Lambda(Z)$ is surjective for each $Z\in\mathcal{U}$. In fact, since $D\Lambda(Z)$ is
a linear map from the tangent space $T_Z(\Omega)$ of $\Omega$ at $Z$ to the tangent space $T_{\Lambda(Z)}(\mathbb{R}^2)$ of $\mathbb{R}^2$ at $\Lambda(Z)$, it suffices to prove that, for each nonzero $(c_1,c_2)\in\mathbb{R}^2$, there is a tangent vector $\zeta$ in $T_Z(\Omega)$ such that $D\Lambda(Z)\zeta=(c_1,c_2)$. By Lemma~\ref{co123} there is a smooth curve $L(\varepsilon):(-\varepsilon_0,\varepsilon_0)\rightarrow\Omega$ satisfying $L(0)=Z$ and
$$\Lambda(L(\varepsilon))=\Lambda(Z)+(c_1,c_2)\varepsilon+\mathcal{O}(\varepsilon^2).$$
Thus
\begin{equation}\label{cjkg32}
\left.\frac{d\Lambda(L(\varepsilon))}{d\varepsilon}\right|_{\varepsilon=0}=\lim_{\varepsilon\rightarrow0}\frac{\Lambda( L(\varepsilon))-\Lambda(Z)}{\varepsilon}=\lim_{\varepsilon\rightarrow0}\frac{(c_1,c_2)\varepsilon+\mathcal{O}(\varepsilon^2)}{\varepsilon}=(c_1,c_2).
\end{equation}
Besides,
\begin{equation}\label{cjkg32casf}
\left.\frac{d\Lambda(L(\varepsilon))}{d\varepsilon}\right|_{\varepsilon=0}=D\Lambda(Z)L'(0),
\end{equation}
where $L'(0)$ is the derivative of $L(\varepsilon)$ at $\varepsilon=0$. By (\ref{cjkg32}) and (\ref{cjkg32casf}) we have $D\Lambda(Z)L'(0)=(c_1,c_2)$, which implies that we can take $\zeta=L'(0)$ and then get that the derivative $D\Lambda(Z)$ is surjective for each $Z\in\mathcal{U}$.

Finally, by the above arguments and \cite[p.22]{HFFFZ}, $\Lambda^{-1}(0,0)$, i.e., the set $\mathcal{U}_0$ of all vector fields of $\mathcal{U}$ having a figure eight loop characterized by {\bf(H1)} and {\bf (H2)} in $\mathcal{A}$, is a codimension-2 $C^k$ submanifold of $\mathcal{U}$.
\end{proof}

%%%%%%%%%%%%%%%%%%%%%%%%%%%%%%%%%%%%%%%%%%%%%%%%%%%%%%%%%%%%
\section{Proof of Theorem \ref{mainthm}}
\setcounter{equation}{0}
\setcounter{lm}{0}
\setcounter{thm}{0}
\setcounter{rmk}{0}
\setcounter{df}{0}
\setcounter{cor}{0}

To prove Theorem \ref{mainthm}, we give some preliminary lemmas in the following.

\begin{lm}\label{cewr343787}
Consider the maps $\varphi_1(Z^+), \varphi_2(Z^+):\mathcal{U}_2^+\rightarrow\mathbb{R}$ obtained in Lemmas~\ref{slidynamics}, \ref{csin34f} respectively, and the vector field $Z^+(x,y;\alpha)$ given in $(\ref{parasys})$ for $\alpha\in U$.
If $\kappa_1g^+_{\alpha_2}(0,0;0)-\kappa_2g^+_{\alpha_1}(0,0;0)\ne0$, then there exists a neighborhood $U_1\subset U$ of $\alpha=0$ such that
\begin{equation}\label{csk3434}
\beta=(\beta_1,\beta_2)=(\varphi_1(\alpha),\varphi_2(\alpha)):=(\varphi_1(Z^+(x,y;\alpha)),\varphi_2(Z^+(x,y;\alpha)))
\end{equation}
is a local diffeomorphism from $U_1$ to its range $V_1$.
\end{lm}

\begin{proof}
By (\ref{csk3434}) and the definition of $\varphi_1(Z^+)$, $\varphi_1(\alpha)$ satisfies $\varphi_1(0)=0$ and $g^+(\varphi_1(\alpha),0;\alpha)=0$. Thus, applying the Implicit Function Theorem to the equation $g^+(x, 0; \alpha)=0$, we can write $\varphi_1(\alpha)$ around $\alpha=0$ as
\begin{equation}\label{cjsdser}
\begin{aligned}
\varphi_1(\alpha)=-\frac{g^+_{\alpha_1}(0,0;0)}{g^+_x(0,0;0)}\alpha_1-\frac{g^+_{\alpha_2}(0,0;0)}{g^+_x(0,0;0)}\alpha_2+\mathcal{O}(\|\alpha\|^2).
\end{aligned}
\end{equation}

On the other hand, by the linear change of variables
$$
(x,y)\rightarrow(x+\varphi_1(\alpha),y),
$$
we can transform $Z^+(x,y;\alpha)=(f^+(x,y;\alpha),g^+(x,y;\alpha))$ to a new vector field
\begin{equation}\label{cj232sdser}
\widehat Z^+(x,y;\alpha)=(\hat f^+(x,y;\alpha),\hat g^+(x,y;\alpha)):=(f^+(x+\varphi_1(\alpha),y;\alpha),g^+(x+\varphi_1(\alpha),y;\alpha))
\end{equation}
for $\alpha\in U$. In particular, $\widehat Z^+(x,y;0)=Z^+(x,y;0)=Z_0^+$. Then we can define a Poincar\'e map $P(y;\alpha)$ having the $y$-axis near $O$ as the Poincar\'e section by the forward orbits of $\widehat Z^+(x,y;\alpha)$. From Lemma~\ref{csin34f} and its proof, $\varphi_2(\alpha)$ satisfies $\varphi_2(0)=0$ and $P(\varphi_2(\alpha);\alpha)=\varphi_2(\alpha)$. Thus, applying the Implicit Function Theorem to the equation $P(y;\alpha)=y$ and using the result of \cite[p.384]{AAA}, we can write $\varphi_2(\alpha)$ around $\alpha=0$ as
\begin{equation}\label{cerid0f}
\begin{aligned}
\varphi_2(\alpha)=\frac{\hat\kappa_1}{\hat f^+(0,0;0)(1-\hat\lambda(0))}\alpha_1+\frac{\hat\kappa_2}{\hat f^+(0,0;0)(1-\hat\lambda(0))}\alpha_2+\mathcal{O}(\|\alpha\|^2),
\end{aligned}
\end{equation}
where
\begin{equation}\label{485ujvsd}
\begin{aligned}
\hat\kappa_i&:=\int^{T_0}_0\hat\lambda(t)(\hat f^+\hat g^+_{\alpha_i}-\hat g^+\hat f^+_{\alpha_i})(\gamma_0(t);0)dt,\quad i=1,2,\\
\hat\lambda(t)&:=\exp\left(\int^{T_0}_t(\hat f^+_x+\hat g^+_y)(\gamma_0(s);0)ds\right),\qquad t\in[0,T_0].
\end{aligned}
\end{equation}
It follows from $\varphi_1(0)=0$ and (\ref{cj232sdser}) that
\begin{equation}\label{cjwn434343}
\hat f^+(x,y;0)=f^+(x,y;0),\qquad \hat g^+(x,y;0)=g^+(x,y;0), \qquad\hat\lambda(t)=\lambda(t)
\end{equation}
and
\begin{equation}\label{4839cfwe}
\begin{aligned}
\hat f^+_{\alpha_i}(x,y;0)&=f^+_x(x,y;0)\frac{\partial\varphi_1(0)}{\partial\alpha_i}+f^+_{\alpha_i}(x,y;0)\\
&=-\frac{g^+_{\alpha_i}(0,0;0)}{g^+_x(0,0;0)}f^+_x(x,y;0)+f^+_{\alpha_i}(x,y;0),\qquad i=1,2,\\
\hat g^+_{\alpha_i}(x,y;0)&=g^+_x(x,y;0)\frac{\partial\varphi_1(0)}{\partial\alpha_i}+g^+_{\alpha_i}(x,y;0)\\
&=-\frac{g^+_{\alpha_i}(0,0;0)}{g^+_x(0,0;0)}g^+_x(x,y;0)+g^+_{\alpha_i}(x,y;0),\qquad i=1,2.
\end{aligned}
\end{equation}
Thus by (\ref{cerid0f}), (\ref{485ujvsd}), (\ref{cjwn434343}) and (\ref{4839cfwe}) we get
\begin{equation}\label{cjsdser222}
\begin{aligned}
\varphi_2(\alpha)=\frac{\nu g^+_{\alpha_1}(0,0;0)/g^+_x(0,0;0)+\kappa_1}{f^+(0,0;0)(1-\lambda(0))}\alpha_1+\frac{\nu g^+_{\alpha_2}(0,0;0)/g^+_x(0,0;0)+\kappa_2}{f^+(0,0;0)(1-\lambda(0))}\alpha_2+\mathcal{O}(\|\alpha\|^2),
\end{aligned}
\end{equation}
where $\lambda(t)$, $\kappa_1$ and $\kappa_2$ are defined in $(\ref{ceirucew434})$,
\begin{equation}\label{4ivag}
\nu:=\int^{T_0}_0\lambda(t)(g^+ f^+_x-f^+g^+_x)(\gamma_0(t);0)dt.
\end{equation}

Finally, under the assumption of $\kappa_1g^+_{\alpha_2}(0,0;0)-\kappa_2g^+_{\alpha_1}(0,0;0)\ne0$,
$$\left.\det \frac{\partial (\varphi_1(\alpha),\varphi_2(\alpha))}{\partial(\alpha_1,\alpha_2)}\right|_{(\alpha_1,\alpha_2)=(0,0)}
=\frac{\kappa_1g^+_{\alpha_2}(0,0;0)-\kappa_2g^+_{\alpha_1}(0,0;0)}{f^+(0,0;0)g^+_x(0,0;0)(1-\lambda(0))}\ne0.$$
which concludes this lemma by the Implicit Function Theorem.
\end{proof}

According to Lemma~\ref{cewr343787}, we now consider the reparameterize vector field
$$
\widetilde Z(x, y;\beta)=\left\{
\begin{aligned}
  &\widetilde Z^+(x, y;\beta)=(p^+(x,y;\beta),q^+(x,y;\beta))
~~~~&& {\rm if}~(x, y)\in\Sigma^+,\\
  &\widetilde Z^-(x, y;\beta)=(p^-(x,y;\beta),q^-(x,y;\beta))
~~~~&& {\rm if}~(x, y)\in\Sigma^-\\
\end{aligned}
\right.
$$
for $\beta\in V_1$, where
\begin{equation}\label{ckw345g}
\begin{aligned}
(p^+(x,y;\beta),q^+(x,y;\beta))&:=(f^+(x+\beta_1,y;\varphi^{-1}(\beta)),g^+(x+\beta_1,y;\varphi^{-1}(\beta))),\\
(p^-(x,y;\beta),q^-(x,y;\beta))&:=(-f^+(-x-\beta_1,-y;\varphi^{-1}(\beta)),-g^+(-x-\beta_1,-y;\varphi^{-1}(\beta))),
\end{aligned}
\end{equation}
and $\varphi^{-1}(\beta)$ denotes the inverse of (\ref{csk3434}). Note that $\widetilde Z(x, y;\beta)$ is $\mathbb{Z}_2$-symmetric vector field with respect to $(-\beta_1,0)$ and
$\widetilde Z(x, y;0)=Z(x, y;0)=Z_0(x,y)$. Therefore, in order to prove Theorem~\ref{mainthm}, we can equivalently study the bifurcation of $\widetilde Z(x, y;\beta)$ for $\beta\in V_1$, provided that $\widetilde Z(x, y;0)$ satisfies the assumptions of Theorem~\ref{mainthm}. Actually, once the bifurcation diagram of $\widetilde Z(x, y;\beta)$ is established, we only need to make a horizontal translation for all the phase portraits to obtain the bifurcation diagram of $Z(x, y;\varphi^{-1}(\beta))$ in the $\beta$-plane.

From Lemma~\ref{slidynamics} and Lemma~\ref{csin34f}, we can directly obtain the following two lemmas on sliding dynamics, standard cycles and grazing cycles of $\widetilde Z(x, y;\beta)$ for $\alpha\in V_1$.

\begin{lm}\label{csdnewr}
Let $\widetilde Z(x,y;0)$ satisfy the assumptions of Theorem~\ref{mainthm}. For a sufficiently small figure eight annulus neighborhood $\mathcal{A}$ of $\Upsilon_0$ there is a neighborhood $V_2\subset V_1$ of $\beta=0$ such that $\widetilde Z(x,y;\beta)$ satisfies the following properties for $\beta\in V_2$.
\begin{itemize}
\item[{\rm(1)}] If $\beta_1=0$, $\mathcal{A}\cap\Sigma$ is split into two crossing segments by a fold-fold at $O$ satisfying
    $$
    \begin{aligned}
    p^-(0,0;\beta)&=-p^+(0,0;\beta)<0,\\
    q^-(0,0;\beta)&=-q^+(0,0;\beta)=0,\\
    q^-_x(0,0;\beta)&=-q^+_x(0,0;\beta)<0.
    \end{aligned}
    $$
\item[{\rm(2)}] If $\beta_1\ne0$, $\mathcal{A}\cap\Sigma$ is split into two crossing segments and a sliding segment $\{(x,0):\min\{-2\beta_1,0\}<x<\max\{-2\beta_1,0\}\}$ by two regular-folds, $O$ and $(-2\beta_1,0)$, which respectively satisfy
$$p^+(0,0;\beta)>0,\qquad q^+(0,0;\beta)=0,\qquad q^+_x(0,0;\beta)>0,\qquad \beta_1q^-(0,0;\beta)>0$$
and
$$p^-(-2\beta_1,0;\beta)<0,\qquad q^-(-2\beta_1,0;\beta)=0,\qquad q^-_x(-2\beta_1,0)<0,\qquad \beta_1q^+(-2\beta_1,0;\beta)<0.$$
    In addition, if $\beta_1>0$ $($resp. $<0$$)$, the sliding segment is stable $($resp. unstable$)$ and there is a unique pseudo-equilibrium, which lies at $(-\beta_1,0)$ and is a pseudo-saddle.
\end{itemize}
\end{lm}

\begin{lm}\label{standardcy}
Let $\widetilde Z(x,y;0)$ satisfy the assumptions of Theorem~\ref{mainthm}. For a sufficiently small figure eight annulus neighborhood $\mathcal{A}$ of $\Upsilon_0$ there exists a neighborhood $V_3\subset V_2$ of $\beta=0$ such that
for $\beta\in V_3$, $\widetilde Z(x,y;\beta)$ has exactly two grazing cycles in $\mathcal{A}$, which are internally stable and $\mathbb{Z}_2$-symmetric with respect to $(-\beta_1,0)$, if $\beta_2=0$. The grazing cycle in $\Sigma^+$ $($resp. $\Sigma^-)$ grazes at $O$ $($resp. $(-2\beta_1,0))$. Moreover, the two grazing cycles become hyperbolically stable standard cycles in $\mathcal{A}$ if $\beta_2>0$, while if $\beta_2<0$, they disappear and there exist no standard cycles and grazing cycles in $\mathcal{A}$.
\end{lm}

To identify the sliding cycles, crossing cycles and sliding homoclinic orbits of $\widetilde Z(x,y;\beta)$,
we introduce two transition maps and a displacement map as follows.
Let $\Pi:=\{(x, b): |x-a|<\delta\}$, where $(a, b)\in\Gamma_0$ satisfies
$$
q^+(a,b;0)=g^+(a, b;0)>0
$$
and $\delta>0$ is a constant such that $(a, b)$ is the unique intersection between $\Pi$ and $\Gamma_0$. Then, according to $({\bf H1})$ and Lemma~\ref{csdnewr}, for a sufficiently small constant $\epsilon_0>0$ there is a neighborhood $V_4\subset V_3$ of $\beta=0$ such that the forward (resp. backward) orbit of $\widetilde Z^+(x,y;\beta)$ for $\beta\in V_4$ with the initial value $(x, 0)$ satisfying $x\in(-\epsilon_0,\epsilon_0)$ can reach $\Pi$ at a point $(x^+, b)$ (resp. $(x^-, b)$) after a finite time $t=\tau^+(x;\beta)>0$ (resp. $t=\tau^-(x;\beta)<0$). Therefore, we can define transition maps
\begin{equation}\label{37ndjks2}
\begin{aligned}
\mathcal{D}^+(x;\beta):=x^+,\qquad \mathcal{D}^-(x;\beta):=x^-
\end{aligned}
\end{equation}
for $x\in(-\epsilon_0,\epsilon_0)$ and $\beta\in V_4$.
Besides, we can select $\epsilon_1\in(0,\epsilon_0)$ and a suitable $V_4$ such that the displacement map
\begin{equation}\label{ckjrie562}
\mathcal{D}(x;\beta):=\mathcal{D}^+(x;\beta)-\mathcal{D}^-(-2\beta_1-x;\beta)
\end{equation}
is defined well for $x\in(-\epsilon_1,\epsilon_1)$ and $\beta\in V_4$.

Based on the definitions of these maps and the sliding dynamics given in Lemma~\ref{csdnewr}, we easily obtain
\begin{lm}\label{hdcshu}
Considering the maps $\mathcal{D}^\pm(x;\beta):(-\epsilon_0,\epsilon_0)\times V_4\rightarrow\mathbb{R}$ and $\mathcal{D}(x;\beta): (-\epsilon_1,\epsilon_1)\times V_4\rightarrow\mathbb{R}$ constructed in $(\ref{37ndjks2})$ and $(\ref{ckjrie562})$ respectively, we have the following statements for $\beta\in V_4$.
\begin{itemize}
\item[{\rm(1)}] The crossing cycles of $\widetilde Z(x,y;\beta)$ bifurcating from $\Upsilon_0$ are in one-to-one correspondence with the zeros of $\mathcal{D}(x;\beta)$ in $\mathcal{I}^o:=(\max\{-2\beta_1,0\},\epsilon_1)$. Moreover, the multiplicity and stability of a crossing cycle are the same as the multiplicity and stability of the corresponding zero of $\mathcal{D}(x;\beta)$.
\item[{\rm(2)}] $\widetilde Z(x,y;\beta)$ has a critical crossing cycle bifurcating from $\Upsilon_0$ if and only if either $\beta_1>0$ and $\mathcal{D}(0;\beta)=0$ or $\beta_1<0$ and $\mathcal{D}(-2\beta_1;\beta)=0$. Moreover, if there is a critical crossing cycle, then it is unique and its external stability is the same as the the right-side stability of the zero $x=\max\{-2\beta_1,0\}$ of $\mathcal{D}(x;\beta)$.
\item[{\rm(3)}] $\widetilde Z(x,y;\beta)$ has a one-zonal sliding cycle bifurcating from $\Upsilon_0$ if and only if either $\beta_1>0, \beta_2<0$ and $\mathcal{D}^+(0;\beta)-\mathcal{D}^-(-\beta_1;\beta)<0$ or $\beta_1<0, \beta_2>0$ and $\mathcal{D}^+(-\beta_1;\beta)-\mathcal{D}^-(0;\beta)>0$. Moreover, if there is a one-zonal sliding cycle, then $\widetilde Z(x,y;\beta)$ has exactly two one-zonal sliding cycles, which are $\mathbb{Z}_2$-symmetric with respect to $(-\beta_1,0)$ and stable $($resp. unstable$)$ for $\beta_1>0$ $($resp. $<0$$)$.
\item[{\rm(4)}] $\widetilde Z(x,y;\beta)$ has a sliding homoclinic orbit bifurcating from $\Upsilon_0$ if and only if either $\beta_1>0, \beta_2<0$ and $\mathcal{D}^+(0;\beta)-\mathcal{D}^-(-\beta_1;\beta)=0$ or $\beta_1<0, \beta_2>0$ and $\mathcal{D}^+(-\beta_1;\beta)-\mathcal{D}^-(0;\beta)=0$. Moreover, if there is a sliding homoclinic orbits, then $\widetilde Z(x,y;\beta)$ has exactly two sliding homoclinic orbits, which are to the pseudo-saddle $(-\beta_1,0)$, one-zonal and $\mathbb{Z}_2$-symmetric with respect to $(-\beta_1,0)$.
\item[{\rm(5)}] $\widetilde Z(x,y;\beta)$ has a two-zonal sliding cycle bifurcating from $\Upsilon_0$ if and only if either $\beta_1>0, \beta_2<0, \mathcal{D}(0;\beta)<0, \mathcal{D}^+(0;\beta)-\mathcal{D}^-(-\beta_1;\beta)>0$ or $\beta_1<0, \beta_2>0, \mathcal{D}(-2\beta_1;\beta)>0, \mathcal{D}^+(-\beta_1;\beta)-\mathcal{D}^-(0;\beta)<0$. Moreover, if there is a two-zonal sliding cycle, then it is unique and stable $($resp. unstable$)$ for $\beta_1>0$ $($resp. $<0$$)$.
\end{itemize}
\end{lm}

By Lemma~\ref{hdcshu}, we only need to study the maps $\mathcal{D}^\pm$ and $\mathcal{D}$ to obtain the information on sliding cycles, crossing cycles and sliding homoclinic orbits bifurcating from $\Upsilon_0$.

\begin{lm}
The maps $\mathcal{D}^+(x;\beta)$ and $\mathcal{D}^-(x;\beta): (-\epsilon_0,\epsilon_0)\times V_4\rightarrow\mathbb{R}$ constructed in $(\ref{37ndjks2})$ can be expanded as
\begin{equation}\label{clksmf23}
\begin{aligned}
\mathcal{D}^+(x;\beta)=&~a+A_1^+\beta_1+A_2^+\beta_2+\mathcal{O}(\|\beta\|^2)+\mathcal{O}(\|\beta\|^2)x
+(B^++\mathcal{O}(\|\beta\|))x^2+\mathcal{O}(x^3),\\
\mathcal{D}^-(x;\beta)=&~a+A_1^-\beta_1+A_2^-\beta_2+\mathcal{O}(\|\beta\|^2)+\mathcal{O}(\|\beta\|^2)x
+(B^-+\mathcal{O}(\|\beta\|))x^2+\mathcal{O}(x^3),
\end{aligned}
\end{equation}
where
\begin{equation}\label{ceiouncs34}
\begin{aligned}
A^\pm_1&:=\frac{\nu^\pm}{g^+(a,b;0)}-\frac{\kappa^\pm_1\left(\nu g^+_{\alpha_2}(0,0;0)+\kappa_2g^+_x(0,0;0)\right)-\kappa^\pm_2\left(\nu g^+_{\alpha_1}(0,0;0)+\kappa_1g^+_x(0,0;0)\right)}{g^+(a,b;0)\left(\kappa_1g^+_{\alpha_2}(0,0;0)-\kappa_2g^+_{\alpha_1}(0,0;0)\right)},\\
A^\pm_2&:=\frac{\left(\kappa^\pm_2g^+_{\alpha_1}(0,0;0)-\kappa^\pm_1g^+_{\alpha_2}(0,0;0)\right)f^+(0,0;0)(1-\lambda(0))}
{g^+(a,b;0)\left(\kappa_1g^+_{\alpha_2}(0,0;0)-\kappa_2g^+_{\alpha_1}(0,0;0)\right)},\\
B^\pm&:=\frac{g^+_x(0,0;0)\lambda^\pm(0)}{2g^+(a, b;0)},\\
\kappa^\pm_i&:=\int^{\tau^\pm_0}_0\lambda^\pm(t)\left(f^+g^+_{\alpha_i}-g^+f^+_{\alpha_i}\right)(\gamma_0(t);0)dt, \qquad i=1, 2,\\
\nu^\pm&:=\int^{\tau^\pm_0}_0\lambda^\pm(t)(g^+f^+_x-f^+g^+_x)(\gamma_0(t),0)dt,\\
\lambda^\pm(t)&:=\exp\left(\int^{\tau^\pm_0}_t\left(f^+_x+g^+_y\right)(\gamma_0(s);0)ds\right),\qquad t\in[0, \tau^\pm_0],
\end{aligned}
\end{equation}
$\lambda(t)$, $\kappa_1$ and $\kappa_2$ are defined in $(\ref{ceirucew434})$, $\nu$ is defined in $(\ref{4ivag})$, $\tau^+_0>0$ $($resp. $\tau^-_0<0$$)$ is the travelling time of the forward $($resp. backward$)$ orbit of $\Gamma_0$ from $O$ to $\Pi$.
\end{lm}

\begin{proof}
Tracing the proof of \cite[Lemma 5.3]{LTXC} and using the fact that $q^+(0,0;\beta)=0$ for $\beta\in V_4$ (see Lemma~\ref{csdnewr}), we get
\begin{equation}\label{eiafaowr34}
\begin{aligned}
\frac{\partial \mathcal{D}^\pm(0;0)}{\partial \beta_i}&=\frac{-1}{q^+(a,b;0)}\int^{\tau_0^\pm}_0\exp\int^{\tau^\pm_0}_t(p^+_x+q^+_y)(\gamma_0(s);0)ds
\Big(p^+q^+_{\beta_i}-q^+p^+_{\beta_i}\Big)(\gamma_0(t);0)dt,\quad i=1,2,\\
\frac{\partial \mathcal{D}^\pm(0;0)}{\partial x}&=\frac{q^+(0,0;0)}{q^+(a,b;0)}\exp\int^{\tau^\pm_0}_0(p^+_x+q^+_y)(\gamma_0(t);0)dt=0,\\
\frac{\partial^2 \mathcal{D}^\pm(0;0)}{\partial x\partial\beta_i}&=\frac{q^+_{\beta_i}(0,0;0)}{q^+(a,b;0)}\exp\int^{\tau^\pm_0}_0(p^+_x+q^+_y)(\gamma_0(t);0)dt=0,\quad i=1,2,\\
\frac{\partial^2 \mathcal{D}^\pm(0;0)}{\partial x^2}&=\frac{q^+_x(0,0;0)}{q^+(a,b;0)}\exp\int^{\tau^\pm_0}_0(p^+_x+q^+_y)(\gamma_0(t);0)dt.
\end{aligned}
\end{equation}
Besides, it follows from (\ref{cjsdser}) and (\ref{cjsdser222}) that
\begin{equation}\label{eiowr34}
\begin{aligned}
\frac{\partial\alpha_1(0)}{\partial\beta_1}&=\frac{\nu g^+_{\alpha_2}(0,0;0)+\kappa_2g^+_x(0,0;0)}{\kappa_1g^+_{\alpha_2}(0,0;0)-\kappa_2g^+_{\alpha_1}(0,0;0)},\\
\frac{\partial\alpha_1(0)}{\partial\beta_2}&=\frac{f^+(0,0;0)(1-\lambda(0))g^+_{\alpha_2}(0,0;0)}{\kappa_1g^+_{\alpha_2}(0,0;0)-\kappa_2g^+_{\alpha_1}(0,0;0)},\\
\frac{\partial\alpha_2(0)}{\partial\beta_1}&=-\frac{\nu g^+_{\alpha_1}(0,0;0)+\kappa_1g^+_x(0,0;0)}{\kappa_1g^+_{\alpha_2}(0,0;0)-\kappa_2g^+_{\alpha_1}(0,0;0)},\\
\frac{\partial\alpha_2(0)}{\partial\beta_2}&=-\frac{f^+(0,0;0)(1-\lambda(0))g^+_{\alpha_1}(0,0;0)}{\kappa_1g^+_{\alpha_2}(0,0;0)-\kappa_2g^+_{\alpha_1}(0,0;0)},
\end{aligned}
\end{equation}
and from (\ref{ckw345g}) that
\begin{equation}\label{aweq3r34}
\begin{aligned}
p^+(x,y;0)&=f^+(x,y;0),\qquad \qquad\qquad q^+(x,y;0)=g^+(x,y;0),\\ p^+_{\beta_1}(x,y;0)&=f^+_x(x,y;0)+f^+_{\alpha_1}(x,y;0)\frac{\partial\alpha_1(0)}{\partial\beta_1}+f^+_{\alpha_2}(x,y;0)\frac{\partial\alpha_2(0)}{\partial\beta_1},\\
q^+_{\beta_1}(x,y;0)&=g^+_x(x,y;0)+g^+_{\alpha_1}(x,y;0)\frac{\partial\alpha_1(0)}{\partial\beta_1}+g^+_{\alpha_2}(x,y;0)\frac{\partial\alpha_2(0)}{\partial\beta_1},\\
p^+_{\beta_2}(x,y;0)&=f^+_{\alpha_1}(x,y;0)\frac{\partial\alpha_1(0)}{\partial\beta_2}+f^+_{\alpha_2}(x,y;0)\frac{\partial\alpha_2(0)}{\partial\beta_2},\\
q^+_{\beta_2}(x,y;0)&=g^+_{\alpha_1}(x,y;0)\frac{\partial\alpha_1(0)}{\partial\beta_2}+g^+_{\alpha_2}(x,y;0)\frac{\partial\alpha_2(0)}{\partial\beta_2}.
\end{aligned}
\end{equation}
Hence, substituting (\ref{eiowr34}) and (\ref{aweq3r34}) into (\ref{eiafaowr34}), we get (\ref{clksmf23}).
\end{proof}

\begin{lm}
Let $A^\pm_1, A^\pm_2$ and $B^\pm$ be given in $(\ref{ceiouncs34})$. Then $A^+_1-A^-_1=0$ and
\begin{equation}\label{cewrjcsdk}
A^+_2-A^-_2=\frac{\lambda^-(0)(\lambda(0)-1)f^+(0,0;0)}{g^+(a,b;0)},\qquad B^+-B^-=\frac{\lambda^-(0)(\lambda(0)-1)g^+_x(0,0;0)}{2g^+(a, b;0)}.
\end{equation}
Moreover,
\begin{equation}\label{cksdj4534}
B^+>0,\qquad B^->0,\qquad A^+_2-A^-_2<0,\qquad B^+-B^-<0,
\end{equation}
under the assumptions of Theorem~\ref{mainthm}.
\end{lm}

\begin{proof}
Due to $\tau^+_0=\tau^-_0+T_0$, by the change $\tilde t=t+T_0$ and $\tilde s=s+T_0$,
\begin{equation}\label{cdnf45t54}
\begin{aligned}
\lambda^-(0)&=\exp\left(\int^{\tau^-_0}_{0}\left(f^+_x+g^+_y\right)(\gamma_0(t);0)dt\right)\\
&=\exp\left(\int^{\tau^+_0}_{T_0}\left(f^+_x+g^+_y\right)(\gamma_0(\tilde t);0)d\tilde t\right)\\
&=\exp\left(\int^{\tau^+_0}_{0}\left(f^+_x+g^+_y\right)(\gamma_0(\tilde t);0)d\tilde t-\int^{T_0}_{0}\left(f^+_x+g^+_y\right)(\gamma_0(\tilde t);0)d\tilde t\right)\\
&=\lambda^+(0)/\lambda(0)
\end{aligned}
\end{equation}
and
$$
\begin{aligned}
\kappa^-_i&=\int^{\tau^-_0}_0\exp\left(\int^{\tau^-_0}_t\left(f^+_x+g^+_y\right)(\gamma_0(s);0)ds\right)
\left(f^+g^+_{\alpha_i}-g^+f^+_{\alpha_i}\right)(\gamma_0(t);0)dt\\
&=\int^{\tau^-_0+T_0}_{T_0}\exp\left(\int^{\tau^-_0+T_0}_{\tilde t}\left(f^+_x+g^+_y\right)(\gamma_0(\tilde s);0)d\tilde s\right)\left(f^+g^+_{\alpha_i}-g^+f^+_{\alpha_i}\right)(\gamma_0(\tilde t);0)d\tilde t\\
&=\int^{\tau^+_0}_{T_0}\exp\left(\int^{\tau^+_0}_{\tilde t}\left(f^+_x+g^+_y\right)(\gamma_0(\tilde s);0)d\tilde s\right)\left(f^+g^+_{\alpha_i}-g^+f^+_{\alpha_i}\right)(\gamma_0(\tilde t);0)d\tilde t\\
&=\kappa^+_i-\int^{T_0}_0\exp\left(\int^{\tau^+_0}_{\tilde t}\left(f^+_x+g^+_y\right)(\gamma_0(\tilde s);0)d\tilde s\right)\left(f^+g^+_{\alpha_i}-g^+f^+_{\alpha_i}\right)(\gamma_0(\tilde t);0)d\tilde t\\
&=\kappa^+_i-\kappa_i\exp\left(\int^{\tau^+_0}_{T_0}\left(f^+_x+g^+_y\right)(\gamma_0(\tilde s);0)d\tilde s\right)\\
&=\kappa^+_i-\kappa_i\exp\left(\int^{\tau^-_0}_{0}\left(f^+_x+g^+_y\right)(\gamma_0(s);0)ds\right)\\
&=\kappa^+_i-\kappa_i\lambda^-(0),\qquad i=1,2,
\end{aligned}
$$
and similarly, $\nu^-=\nu^+-\nu\lambda^-(0)$,
where $\kappa_i$ and $\lambda(t)$ are defined in (\ref{ceirucew434}), $\nu$ is defined in (\ref{4ivag}), $\kappa^\pm_i$ and $\lambda^\pm(t)$ are defined in (\ref{ceiouncs34}).
Hence, a standard computation yields $A^+_1-A^-_1=0$ and (\ref{cewrjcsdk}). Besides, under the assumptions of Theorem~\ref{mainthm}, we have
$$g^+_x(0,0;0)>0,\qquad f^+(0,0;0)>0,\qquad \lambda(0)<1,$$
which together with $\lambda^\pm(0)>0$ and $g^+(a,b;0)>0$ imply (\ref{cksdj4534}).
\end{proof}

\begin{lm}\label{ceriucs}
Let $\mathcal{D}(x;\beta): (-\epsilon_1,\epsilon_1)\times V_4\rightarrow\mathbb{R}$ be the map defined in $(\ref{ckjrie562})$. There exists a neighborhood $V_5\subset V_4$ of $\beta=0$ and a smooth function
\begin{equation}\label{ew84sd}
\psi_1(\beta_1)=\frac{2\lambda(0)g^+_x(0,0;0)}{(\lambda(0)-1)^2f^+(0,0;0)}\beta_1^2+\mathcal{O}(\beta_1^3)
\end{equation}
defined for $(\beta_1,\psi_1(\beta_1))\in V_5$ such that the following statements hold for $\beta\in V_5$.
\vspace{-14pt}
\begin{itemize}
\setlength{\itemsep}{0mm}
\item[{\rm(1)}] If $\beta_2>\psi_1(\beta_1)$, $\mathcal{D}(x;\beta)$ has no zeros in $(-\epsilon_1,\epsilon_1)$.
\item[{\rm(2)}] If $\beta_2=\psi_1(\beta_1)$, $\mathcal{D}(x;\beta)$ has exactly one zero $x_*$ in $(-\epsilon_1,\epsilon_1)$, which is of multiplicity two with $\partial \mathcal{D}(x_*;\beta)/\partial x=0$ and $\partial^2\mathcal{D}(x_*;\beta)/\partial x^2<0$.
\item[{\rm(3)}] If $\beta_2<\psi_1(\beta_1)$, $\mathcal{D}(x;\beta)$ has exactly two zeros $x_+$ and $x_-$ in $(-\epsilon_1,\epsilon_1)$, which are simple with $\partial \mathcal{D}(x_+;\beta)/\partial x<0$ and $\partial \mathcal{D}(x_-;\beta)/\partial x>0$.
\end{itemize}
\end{lm}

\begin{proof}
From Lemmas~\ref{csdnewr} and \ref{standardcy},
\begin{equation}\label{894uncds}
\mathcal{D}^+(0;\beta)=\mathcal{D}^-(0;\beta)\qquad {\rm if}\quad \beta_2=0.
\end{equation}
This together with (\ref{clksmf23}) concludes
\begin{equation}\label{ei34vm90}
\begin{aligned}
\mathcal{D}(x;\beta)=&~\mathcal{D}^+(x;\beta)-\mathcal{D}^-(-2\beta_1-x;\beta)\\
=&~(A_2^+-A_2^-+\mathcal{O}(\|\beta\|))\beta_2+\mathcal{O}(\|\beta\|^2)x-\mathcal{O}(\|\beta\|^2)(-2\beta_1-x)\\
&+(B^++\mathcal{O}\left(\|\beta\|\right))x^2-(B^-+\mathcal{O}\left(\|\beta\|\right))(-2\beta_1-x)^2+\mathcal{O}(x^3)
+\mathcal{O}((-2\beta_1-x)^3)\\
=&~(A_2^+-A_2^-+\mathcal{O}(\|\beta\|))\beta_2-(4B^-\beta_1+\mathcal{O}(\|\beta\|^2))\beta_1-(4B^-\beta_1
+\mathcal{O}(\|\beta\|^2))x\\
&+(B^+-B^-+\mathcal{O}(\|\beta\|))x^2+\mathcal{O}(x^3).
\end{aligned}
\end{equation}

Taking a linear shift
$\tilde x=I(x;\beta):=x+\rho,$
where $\rho=\rho(\beta)$ is a priori unknown function that will be determined later, from (\ref{ei34vm90}) we get
\begin{equation}\label{euwcd}
\begin{aligned}
\mathcal{D}\circ I^{-1}(\tilde x;\beta)=&~(A^+_2-A^-_2+\mathcal{O}(\|\beta\|))\beta_2-(4B^-\beta_1+\mathcal{O}(\|\beta\|^2))\beta_1
-(4B^-\beta_1+\mathcal{O}(\|\beta\|^2))(\tilde x-\rho)\\
&+(B^+-B^-+\mathcal{O}(\|\beta\|))(\tilde x-\rho)^2+\mathcal{O}((\tilde x-\rho)^3)\\
=&~(A^+_2-A^-_2+\mathcal{O}(\|\beta\|))\beta_2-(4B^-\beta_1+\mathcal{O}(\|\beta\|^2))\beta_1
+(4B^-\beta_1+\mathcal{O}(\|\beta\|^2))\rho\\
&+(B^+-B^-+\mathcal{O}(\|\beta\|))\rho^2+\mathcal{O}(\rho^3)
-\Big\{4B^-\beta_1+\mathcal{O}(\|\beta\|^2)\\
&+2(B^+-B^-+\mathcal{O}(\|\beta\|))\rho+\mathcal{O}(\rho^2)\Big\}\tilde x
+(B^+-B^-+\mathcal{O}(\|\beta\|)+\mathcal{O}(\|\rho\|))\tilde x^2+\mathcal{O}(\tilde x^3).
\end{aligned}
\end{equation}
Let $H(\beta,\rho)$ be the linear term of (\ref{euwcd}) with respect to $\tilde x$. Then
$$H(0,0)=0,\quad \frac{\partial H(0,0)}{\partial\rho}=-2(B^+-B^-),\quad \frac{\partial H(0,0)}{\partial\beta_1}=-4B^-,\quad \frac{\partial H(0,0)}{\partial\beta_2}=0.$$
Due to $B^+-B^-<0$ in (\ref{cksdj4534}), there exists a neighborhood $V_{51}\subset V_4$ of $\beta=0$ and a smooth function
\begin{equation}\label{ceiwruncs}
\rho(\beta)=-\frac{2B^-}{B^+-B^-}\beta_1+\mathcal{O}(\|\beta\|^2)
\end{equation}
defined in $V_{51}$ such that $\rho(0)=0$ and $H(\beta,\rho(\beta))=0$ by the Implicit Function Theorem.

Substituting (\ref{ceiwruncs}) into (\ref{euwcd}), we obtain
$$
\begin{aligned}
\mathcal{D}\circ I^{-1}(\tilde x;\beta)
=(A^+_2-A^-_2+\mathcal{O}(\|\beta\|))\beta_2-\frac{4B^+B^-}{B^+-B^-}\beta_1^2+\mathcal{O}(\|\beta\|^3)+(B^+-B^-+\mathcal{O}(\|\beta\|))\tilde x^2+\mathcal{O}(\tilde x^3).
\end{aligned}
$$
Regarding $\tilde x^2$ as a new variable and applyng the Implicit Function Theorem, there exists a neighborhood $V_{52}\subset V_{51}$ of $\beta=0$ and a smooth function
\begin{equation}\label{12idkvfdgr}
\begin{aligned}
\varpi(\beta)&=-\frac{A^+_2-A^-_2}{B^+-B^-}\beta_2+\frac{4B^+B^-}{(B^+-B^-)^2}\beta_1^2+\mathcal{O}(\|\beta\|)\beta_2+\mathcal{O}(\|\beta\|)\beta_1^2\\
\end{aligned}
\end{equation}
defined in $V_{52}$ such that $\mathcal{D}\circ I^{-1}(\tilde x;\beta)$ has no zeros in $(-\epsilon_1,\epsilon_1)$ if $\varpi(\beta)<0$,
a unique zero $\tilde x_*(\beta)=0$ in $(-\epsilon_1,\epsilon_1)$ if $\varpi(\beta)=0$, and exactly two zeros
$\tilde x_\pm(\beta)=\pm\sqrt{\varpi(\beta)}$ in $(-\epsilon_1,\epsilon_1)$ if $\varpi(\beta)>0$.

Consider the equation $\varpi(\beta)=0$. It follows from (\ref{cksdj4534}) and the Implicit Function Theorem that there exists a neighborhood $V_{5}\subset V_{52}$ of $\beta=0$ and a smooth function $\psi_1(\beta_1)$ defined for $(\beta_1,\psi_1(\beta_1))\in V_{5}$ such that $\psi_1(0)=0$ and $\varpi(\beta_1,\psi_1(\beta_1))=0$. In addition, the expansion of $\psi_1(\beta_1)$ around $\beta_1=0$ can be written in the form (\ref{ew84sd}) from (\ref{ceiouncs34}), (\ref{cewrjcsdk}) and (\ref{cdnf45t54}).
Since $(A^+_2-A^-_2)/(B^+-B^-)>0$ from (\ref{cksdj4534}), $\varpi(\beta)<0$ (resp. $>0$) is equivalent to $\beta_2>\psi_1(\beta_1)$ (resp. $<\psi_1(\beta_1)$).
Thus $\mathcal{D}(x;\beta)$ has no zeros in $(-\epsilon_1,\epsilon_1)$ if $\beta_2>\psi_1(\beta_1)$, a unique zero
\begin{equation}\label{ckkfjew}
x_*=\tilde x_*-\rho(\beta)=-\rho(\beta)
\end{equation}
in $(-\epsilon_1,\epsilon_1)$ if $\beta_2=\psi_1(\beta_1)$, and exactly two zeros
\begin{equation}\label{cw32kfjew}
x_\pm=\tilde x_\pm-\rho(\beta)=\pm\sqrt{\varpi(\beta)}-\rho(\beta)
\end{equation}
in $(-\epsilon_1,\epsilon_1)$ if $\beta_2<\psi_1(\beta_1)$.
Finally, due to $\partial I^{-1}(\tilde x;\beta)/\partial \tilde x=1$ and $B^+-B^-<0$ in (\ref{cksdj4534}),
$$
\begin{aligned}
&\frac{\partial \mathcal{D}(x_*;\beta)}{\partial x}=0,\qquad \frac{\partial^2 \mathcal{D}(x_*;\beta)}{\partial x^2}=2(B^+-B^-)+\mathcal{O}(\|\beta\|)<0,\\ &\frac{\partial \mathcal{D}(x_\pm;\beta)}{\partial x}=\pm\Big(2(B^+-B^-)+\mathcal{O}(\|\beta\|)\Big)\sqrt{\varpi(\beta)}+\mathcal{O}(\varpi(\beta))\lessgtr0.
\end{aligned}
$$
The proof of Lemma~\ref{ceriucs} is completed.
\end{proof}

In the following lemma we determine the values of $\beta$ for which the zeros of $\mathcal{D}(x;\beta)$ belong to $\mathcal{I}:=[\max\{-2\beta_1,0\},\epsilon_1)$.
\begin{lm}\label{98hcsdj}
Consider the map $\mathcal{D}(x;\beta): (-\epsilon_1,\epsilon_1)\times V_4\rightarrow\mathbb{R}$ constructed in $(\ref{ckjrie562})$ and its zeros $x_*$ existing for $\beta_2=\psi_1(\beta_1)$ and $x_\pm$ existing for $\beta_2<\psi_1(\beta_1)$ given in $(\ref{ckkfjew})$ and $(\ref{cw32kfjew})$ respectively. Then there is a neighborhood $V_6\subset V_5$ of $\beta=0$ such that following statements hold for $\beta\in V_6$.
\vspace{-14pt}
\begin{itemize}
\setlength{\itemsep}{0mm}
\item[{\rm(1)}] $x_*\notin\mathcal{I}$ if $\beta_1>0$, $x_*\in\mathcal{I}^o$ if $\beta_1<0$, and $x_*=0\in\mathcal{I}$ if $\beta_1=0$.
\item[{\rm(2)}] There exists a smooth function
 \begin{equation}\label{csk35454}
 \psi_2(\beta_1)=-\frac{2\lambda(0)g^+_x(0,0;0)}{(\lambda(0)-1)f^+(0,0;0)}\beta_1^2+\mathcal{O}(\beta_1^3)
 \end{equation}
  defined for $(\beta_1,\psi_2(\beta_1))\in V_6$ such that $x_-\notin\mathcal{I}$ if $\beta_1<0$ and $\beta_2<\psi_2(\beta_1)$ or $\beta_1\ge0$, $x_-=-2\beta_1\in\mathcal{I}$ if $\beta_1<0$ and $\beta_2=\psi_2(\beta_1)$, and $x_-\in\mathcal{I}^o$ if $\beta_1<0$ and $\beta_2>\psi_2(\beta_1)$.
\item[{\rm(3)}] There exists a smooth function
 \begin{equation}\label{csk3545dafa}
 \psi_3(\beta_1)=\frac{2g^+_x(0,0;0)}{(\lambda(0)-1)f^+(0,0;0)}\beta_1^2+\mathcal{O}(\beta_1^3)
 \end{equation}
defined for $(\beta_1,\psi_3(\beta_1))\in V_6$ such that $x_+\in\mathcal{I}^o$ if $\beta_1>0$ and $\beta_2<\psi_3(\beta_1)$ or $\beta_1\le0$, $x_+=0\in\mathcal{I}$ if $\beta_1>0$ and $\beta_2=\psi_3(\beta_1)$, $x_+\notin\mathcal{I}$ if $\beta_1>0$ and $\beta_2>\psi_3(\beta_1)$.
\end{itemize}
\end{lm}

\begin{proof}
Consider the zero $x_*$ existing for $\beta_2=\psi_1(\beta_1)$. We write
$$x_*=\frac{2}{\lambda(0)-1}\beta_1+\mathcal{O}(\beta_1^2)$$
by (\ref{ceiouncs34}), (\ref{cewrjcsdk}), (\ref{ceiwruncs}) and (\ref{ckkfjew}). If $\beta_1>0$, then $\mathcal{I}=[0,\epsilon_1)$, and it follows from $\lambda(0)<1$ that $x_*<0$, and then $x_*\notin\mathcal{I}$. If $\beta_1<0$, then $\mathcal{I}=[-2\beta_1,\epsilon_1)$, and
it follows from $0<\lambda(0)<1$ that
${2}/(\lambda(0)-1)<-2,$
and then $x_*\in\mathcal{I}^o$ in a sufficiently small neighborhood of $\beta=0$. If $\beta_1=0$, then $\mathcal{I}=[0,\epsilon_1)$ and $x_*=0$, which gives $x_*\in\mathcal{I}$.  That is, statement (1) holds.

Consider the zero $x_\pm$ existing for $\beta_2<\psi_1(\beta_1)$. By (\ref{ceiouncs34}), (\ref{cewrjcsdk}), (\ref{ceiwruncs}), (\ref{12idkvfdgr}) and (\ref{cw32kfjew}) we write
\begin{equation}\label{48sffjds}
x_\pm=\pm\sqrt{-\frac{2f^+(0,0;0)}{g^+_x(0,0;0)}\beta_2+\mathcal{O}(\beta_2^2)}+\mathcal{O}(\beta_2^2)
\end{equation}
for $\beta_1=0$, and
\begin{equation}\label{48fjds}
x_\pm=\pm\left(\frac{2f^+(0,0;0)}{g^+_x(0,0;0)}+\mathcal{O}(\beta_1)\right)^{\frac{1}{2}}\sigma
+\frac{2}{\lambda(0)-1}\beta_1+\mathcal{O}\left(\|\beta_1,\sigma)\|^2\right)
\end{equation}
for $\beta_1\ne0$, where
$$\sigma:=\sqrt{\psi_1(\beta_1)-\beta_2}.$$

If $\beta_1=0$, it is direct to get $x_-<0$ for $\beta_2<0$ sufficiently closed to $0$ from (\ref{48sffjds}).
If $\beta_1>0$, due to $\sigma>0$ and $\lambda(0)<1$, we still have $x_-<0$ in a sufficiently small neighborhood of $(\beta_1,\sigma)=(0,0)$ from (\ref{48fjds}). Thus, in the case of $\beta_1\ge0$, $x_-\notin\mathcal{I}$ due to $\mathcal{I}=[0,\epsilon_1)$.
If $\beta_1<0$, we consider
$$
K_-(\beta_1,\sigma):=x_-+2\beta_1=-\left(\frac{2f^+(0,0;0)}{g^+_x(0,0;0)}+\mathcal{O}(\beta_1)\right)^{\frac{1}{2}}\sigma
+\frac{2\lambda(0)}{\lambda(0)-1}\beta_1+\mathcal{O}\left(\|(\beta_1,\sigma)\|^2\right).
$$
By the Implicit Function Theorem the equation $K_-(\beta_1,\sigma)$ determines a smooth function $\sigma_-(\beta_1)$, defined in a sufficiently small neighborhood of $\beta_1=0$, such that $\sigma_-(0)=0$ and $K_-(\beta_1,\sigma_-(\beta_1))=0$. Moreover,
\begin{equation}\label{68j3iudkscs}
\sigma_-(\beta_1)=\frac{2\lambda(0)}{\lambda(0)-1}\left(\frac{g^+_x(0,0;0)}{2f^+(0,0;0)}\right)^{\frac{1}{2}}\beta_1
+\mathcal{O}(\beta_1^2),
\end{equation}
and $K_-(\beta_1,\sigma)>0$ (resp. $<0$) if $\sigma<\sigma_-(\beta_1)$ (resp. $>\sigma_-(\beta_1)$).
Taking
$\psi_2(\beta_1)=\psi_1(\beta_1)-\sigma_-^2(\beta_1)$,
we can verify that $\psi_2(\beta_1)$ is smooth by the smoothness of $\psi_1(\beta_1)$ and $\sigma_-(\beta_1)$, and that its expansion is (\ref{csk35454})
by (\ref{ew84sd}) and (\ref{68j3iudkscs}).
Moreover, $K_-(\beta_1,\sigma)=x_-+2\beta_1=0$ (resp. $>0, <0$) if $\beta_2-\psi_2(\beta_1)=0$ (resp. $>0, <0$), according to the definitions of $\sigma$ and $\psi_2(\beta_1)$.
Note that $\mathcal{I}=[-2\beta_1,\epsilon_1)$ if $\beta_1<0$. Consequently, the above analysis allows us to choose a neighborhood $V_6\subset V_5$ of $\beta=0$ such that statement (2) holds.

If $\beta_1=0$, then $x_+\in\mathcal{I}^o=(0,\epsilon_1)$ for $\beta_2<0$ sufficiently closed to $0$ from (\ref{48sffjds}). If $\beta_1<0$, it follows from (\ref{48fjds}), $\sigma>0$ and $0<\lambda(0)<1$ that
$$x_++2\beta_1=\left(\frac{2f^+(0,0;0)}{g^+_x(0,0;0)}+\mathcal{O}(\beta_1)\right)^{\frac{1}{2}}\sigma
+\frac{2\lambda(0)}{\lambda(0)-1}\beta_1+\mathcal{O}\left(\|(\beta_1,\sigma)\|^2\right)>0,$$
which implies that $x_+\in\mathcal{I}^0$ due to $\mathcal{I}=[-2\beta_1,\epsilon_1)$.
If $\beta_1>0$, we consider
$$
K_+(\beta_1,\sigma):=x_+=\left(\frac{2f^+(0,0;0)}{g^+_x(0,0;0)}+\mathcal{O}(\beta_1)\right)^{\frac{1}{2}}\sigma
+\frac{2}{\lambda(0)-1}\beta_1+\mathcal{O}\left(\|(\beta_1,\sigma)\|^2\right).
$$
By the Implicit Function Theorem the equation $K_+(\beta_1,\sigma)=0$ determines a smooth function $\sigma_+(\beta_1)$, defined in a sufficiently small neighborhood of $\beta_1=0$, such that $\sigma_+(0)=0$ and $K_+(\beta_1,\sigma_+(\beta_1))=0$. Moreover,
\begin{equation}\label{3iudkscs}
\sigma_+(\beta_1)=-\frac{2}{\lambda(0)-1}\left(\frac{g^+_x(0,0;0)}{2f^+(0,0;0)}\right)^{\frac{1}{2}}\beta_1+
\mathcal{O}(\beta_1^2),
\end{equation}
and $K_+(\beta_1,\sigma)>0$ (resp. $<0$) if $\sigma>\sigma_+(\beta_1)$ (resp. $<\sigma_+(\beta_1)$).
Taking
$\psi_3(\beta_1)=\psi_1(\beta_1)-\sigma_+^2(\beta_1),$
we get that $\psi_3(\beta_1)$ is smooth by the smoothness of $\psi_1(\beta_1)$ and $\sigma_+(\beta_1)$, and that its expansion is (\ref{csk3545dafa})
by (\ref{ew84sd}) and (\ref{3iudkscs}). Moreover,
$K_+(\beta_1,\sigma)=x_+=0$ (resp. $>0, <0$) if $\beta_2-\psi_3(\beta_1)=0$ (resp. $<0, >0$),
according to the definitions of $\sigma$ and $\psi_3(\beta_1)$.
Note that $\mathcal{I}=[0,\epsilon_1)$ if $\beta_1>0$. Consequently, the above analysis allows us to choose a neighborhood $V_6\subset V_5$ of $\beta=0$ such that statement (3) holds.
\end{proof}

According to the statements (3)-(5) of Lemma~\ref{hdcshu}, the next lemma determines the values of $\beta$ for which sliding homoclinic orbits or sliding cycles do exit.

\begin{lm}\label{cdsk3454}
Let
$$
P(\beta_1,\beta_2):=\mathcal{D}^+(-\beta_1;\beta)-\mathcal{D}^-(0;\beta), \qquad Q(\beta_1,\beta_2):=\mathcal{D}^+(0;\beta)-\mathcal{D}^-(-\beta_1;\beta)
$$
for $\beta\in V_4$. Then there exists a neighborhood $V_7\subset V_4$ of $\beta=0$ such that the following statements hold for $\beta\in V_7$.
\vspace{-14pt}
\begin{itemize}
\setlength{\itemsep}{0mm}
\item[{\rm(1)}] There exists a smooth function
\begin{equation}\label{ww9387fvdf}
\psi_4(\beta_1)=-\frac{\lambda(0)g^+_x(0,0;0)}{2(\lambda(0)-1)f^+(0,0;0)}\beta_1^2+\mathcal{O}(\beta_1^3)
\end{equation}
 defined for $(\beta_1,\psi_4(\beta_1))\in V_7$ such that $\psi_4(0)=0$ and $P(\beta_1,\psi_4(\beta_1))=0$. Moreover,
 $P(\beta_1,\beta_2)>0$ $($resp. $<0$$)$ if $\beta_2<\psi_4(\beta_1)$ $($resp. $>\psi_4(\beta_1)$$)$.
\item[{\rm(2)}] There exists a smooth function
\begin{equation}\label{www32irpcds}
\psi_5(\beta_1)=\frac{g^+_x(0,0;0)}{2(\lambda(0)-1)f^+(0,0;0)}\beta_1^2+\mathcal{O}(\beta_1^3)
 \end{equation}
 defined for $(\beta_1,\psi_5(\beta_1))\in V_7$ such that $\psi_5(0)=0$ and $Q(\beta_1,\psi_5(\beta_1))=0$. Moreover,
    $Q(\beta_1,\beta_2)>0$ $($resp. $<0$$)$ if $\beta_2<\psi_5(\beta_1)$ $($resp. $>\psi_5(\beta_1)$$)$.
\end{itemize}
\end{lm}

\begin{proof}
By (\ref{clksmf23}) and (\ref{894uncds}),
$$P(\beta_1,\beta_2)=(A^+_2-A^-_2+\mathcal{O}\left(\|\beta\|\right))\beta_2+(B^++\mathcal{O}\left(\|\beta\|\right))\beta_1^2
+\mathcal{O}(\beta_1^3).$$
Clearly, $P(0,0)=0$ and $\partial P(0,0)/\partial\beta_2=A^+_2-A^-_2<0$ from (\ref{cksdj4534}). Therefore, by the Implicit Function Theorem there exists a neighborhood $V_7\subset V_4$ of $\beta=0$ and a smooth function $\psi_4(\beta_1)$ defined for $(\beta_1,\psi_4(\beta_1))\in V_7$ such that $\psi_4(0)=0$ and $P(\beta_1,\psi_4(\beta_1))=0$. Moreover, we get (\ref{ww9387fvdf}) from (\ref{ceiouncs34}) and (\ref{cewrjcsdk}). Due to $A^+_2-A^-_2<0$, $P(\beta_1,\beta_2)>0$ $($resp. $<0$$)$ if $\beta_2<\psi_4(\beta_1)$ $($resp. $>\psi_4(\beta_1)$$)$. Statement (1) holds.

By (\ref{clksmf23}) and (\ref{894uncds}) again,
$$Q(\beta_1,\beta_2)=(A^+_2-A^-_2+\mathcal{O}\left(\|\beta\|\right))\beta_2-(B^-+\mathcal{O}\left(\|\beta\|\right))\beta_1^2+\mathcal{O}(\beta_1^3).$$
A similar analysis yields that there exists a neighborhood $V_7\subset V_4$ of $\beta=0$ and a smooth function $\psi_5(\beta_1)$ defined for $(\beta_1,\psi_5(\beta_1))\in V_7$ such that $\psi_5(0)=0$ and $Q(\beta_1,\psi_5(\beta_1))=0$. Moreover, we get (\ref{www32irpcds}) from (\ref{ceiouncs34}) and (\ref{cewrjcsdk}) again.
Due to $A^+_2-A^-_2<0$, $Q(\beta_1,\beta_2)>0$ $($resp. $<0$$)$ if $\beta_2<\psi_5(\beta_1)$ $($resp. $>\psi_5(\beta_1)$$)$. Statement (2) holds.
\end{proof}

Now we are in a suitable position to give the proof of Theorem~\ref{mainthm}.

\begin{proof}[{\bf Proof of Theorem~\ref{mainthm}}]
For a sufficiently small figure eight annulus neighborhood $\mathcal{A}$ of $\Upsilon_0$ we let $U^*\subset U$ be the neighborhood of $\alpha=0$ such that $V^*:=(\varphi_1(U^*),\varphi_2(U^*))\subset V_6\cap V_7$, where $(\beta_1,\beta_2)=(\varphi_1(\alpha),\varphi_2(\alpha))$ are given in $(\ref{csk3434})$, $V_6$ and $V_7$ are given in Lemmas~\ref{98hcsdj} and \ref{cdsk3454} respectively. Under the condition (\ref{transversality}), $(\beta_1,\beta_2)=(\varphi_1(\alpha), \varphi_2(\alpha))$ is a diffeomorphism from $U^*$ to $V^*$ by Lemma~\ref{cewr343787}.
Let $\psi_i(\beta_1)$, $i=1,2,3,4,5$, be the functions given in Lemmas~\ref{ceriucs}, \ref{98hcsdj}, \ref{cdsk3454} respectively. Then $\psi_i(\beta_1)$, $i=1,2,3,4,5$, are defined well in $\varphi_1(U^*)$. Moreover, they satisfy (\ref{cewr324fwe4}) and are quadratically tangent to $\beta_2=0$ at $(0,0)$ from the corresponding expansions, since $f^+(0,0;0)>0$, $g^+_x(0,0;0)>0$ and $0<\lambda(0)<1$. Thus the curves $\beta_2=\psi_i(\beta_1)$ for $\beta_1<0$, $\beta_2=\psi_j(\beta_1)$ for $\beta_1>0$, $i=1,2,4$, $j=3,5$, and the lines $\beta_1=0$ and $\beta_2=0$ split $V^*$ into $9$ open regions. For each region and its boundary, the dynamics of $\widetilde Z(x,y;\beta)$ in $\mathcal{A}$ can be obtained by summarizing the following statements (i)-(ix) and the sliding dynamics stated in Lemma~\ref{csdnewr}.

Based on Lemma~\ref{standardcy}, we obtain the information on standard cycles and grazing cycles.
\vspace{-14pt}
\begin{itemize}
\setlength{\itemsep}{0mm}
\item[{\rm(i)}] If $\beta_2>0$, there exist exactly two standard cycles in $\mathcal{A}$, which are hyperbolically stable and $\mathbb{Z}_2$-symmetric with respect to $(-\beta_1,0)$.  If $\beta_2=0$, the two cycles become two internally stable grazing cycles in $\mathcal{A}$ that are $\mathbb{Z}_2$-symmetric with respect to $(-\beta_1,0)$ and the grazing cycle in $\Sigma^+$ grazes at $O$ and the other grazes at $(-2\beta_1,0)$, in particular, they form a figure eight loop if $\beta_1=0$. If $\beta_2<0$, there are no standard cycles and grazing cycles in $\mathcal{A}$.
\end{itemize}
\vspace{-14pt}

Based on (1) and (2) of Lemma~\ref{hdcshu}, Lemmas~\ref{ceriucs} and \ref{98hcsdj}, we obtain the information on crossing cycles and critical crossing cycles.
\vspace{-14pt}
\begin{itemize}
\setlength{\itemsep}{0mm}
\item[{\rm(ii)}] If $\beta_1<0$ and $\beta_2=\psi_1(\beta_1)$, there exists a unique crossing cycle, which is of multiplicity two and stable from the outside.
\item[{\rm(iii)}] If $\beta_1<0$ and $\psi_2(\beta_1)<\beta_2<\psi_1(\beta_1)$, there exist exactly two crossing cycles. The outside crossing cycle is hyperbolic and stable, and the inside one is hyperbolic and unstable. If $\beta_1<0$ and $\beta_2\le\psi_2(\beta_1)$, the outside crossing cycle persists but the inside one becomes an unstable critical crossing cycle from the outside for $\beta_2=\psi_2(\beta_1)$ and disappears for $\beta_2<\psi_2(\beta_1)$.
\item[{\rm(iv)}] If $\beta_1=0$ and $\beta_2<0$, there exists a unique crossing cycle, which is hyperbolic and stable.
\item[{\rm(v)}] If $\beta_1>0$ and $\beta_2<\psi_3(\beta_1)$, there exists a unique crossing cycle, which is hyperbolic and stable. Moreover, the crossing cycle becomes a stable critical crossing cycle from the outside if $\beta_1>0$ and $\beta_2=\psi_3(\beta_1)$ and then disappears if $\beta_1>0$ and $\beta_2>\psi_3(\beta_1)$.
\item[{\rm(vi)}] There exist no crossing cycles and critical crossing cycles for other values of parameters.
\end{itemize}
\vspace{-14pt}

Based on (3)-(5) of Lemma~\ref{hdcshu} and Lemma~\ref{cdsk3454}, we obtain the information on sliding cycles and sliding homoclinic orbits.
\vspace{-14pt}
\begin{itemize}
\setlength{\itemsep}{0mm}
\item[{\rm(vii)}] If $\beta_1>0$ and $\psi_5(\beta_1)<\beta_2<0$, there exist exactly two sliding cycles, which are stable, one-zonal and $\mathbb{Z}_2$-symmetric with respect to $(-\beta_1,0)$. They become two sliding homoclinic orbits to the pseudo-saddle $(-\beta_1,0)$ if $\beta_1>0$ and $\beta_2=\psi_5(\beta_1)$ and then become a stable two-zonal sliding cycle if $\beta_1>0$ and $\psi_3(\beta_1)<\beta_2<\psi_5(\beta_1)$.
\item[{\rm(viii)}] If $\beta_1<0$ and $0<\beta_2<\psi_4(\beta_1)$, there exist exactly two sliding cycles, which are unstable, one-zonal and $\mathbb{Z}_2$-symmetric with respect to $(-\beta_1,0)$. They become two sliding homoclinic orbits to the pseudo-saddle $(-\beta_1,0)$ if $\beta_1<0$ and $\beta_2=\psi_4(\beta_1)$ and then become an unstable two-zonal sliding cycle if $\beta_1>0$ and $\psi_4(\beta_1)<\beta_2<\psi_2(\beta_1)$.
\item[{\rm(ix)}] There exist no sliding cycles and sliding homoclinic orbits for other values of parameters.
\end{itemize}
\vspace{-14pt}

Eventually, we can horizontally translate all the phase portraits in the bifurcation diagram of $\widetilde Z(x,y;\alpha)$ to obtain the bifurcation diagram of $Z(x, y;\alpha)$ in $\beta$-plane. The proof of Theorem~\ref{mainthm} is finished.
\end{proof}

%%%%%%%%%%%%%%%%%%%%%%%%%%%%%%%%%%%%%%%%%%%%%%%%%%%%%%%%%%%%
\section{Example}
\setcounter{equation}{0}
\setcounter{lm}{0}
\setcounter{thm}{0}
\setcounter{rmk}{0}
\setcounter{df}{0}
\setcounter{cor}{0}

In this section we show an example to realize the bifurcation described in Theorem~\ref{mainthm}.
Consider the following system
\begin{equation}\label{wdafddsdjnf}
\left(
\begin{array}{c}
\dot x\\
\dot y
\end{array}
\right)=
\left\{
\begin{aligned}
&\left(
\begin{array}{c}
y-(ax+bx^3)+\alpha_1+\alpha_2(x-x^2)\\
1-x
\end{array}
\right)\qquad &&{\rm for}\quad x>0,\\
&\left(
\begin{array}{c}
y-(ax+bx^3)-\alpha_1+\alpha_2(x+x^2)\\
-1-x
\end{array}
\right)\qquad &&{\rm for}\quad x<0.
\end{aligned}
\right.
\end{equation}
If $\alpha_1=\alpha_2=0$, system (\ref{wdafddsdjnf}) is the discontinuous limit case of a smooth oscillator
introduced in \cite{CWPGT}, which is derived from an archetypal system by Thompson and Hunt
\cite{TH} and is widely used in engineering.

Clearly, system (\ref{wdafddsdjnf}) is $\mathbb{Z}_2$-symmetric with respect to $O$. If $\alpha_1=\alpha_2=0$, it is easy to verify that $O$ is a visible fold-fold and, based on the results of \cite{CHB}, there exists a smooth function $\vartheta(a)$ in $(-\infty,0)$ such that for $a<0$ and $b=\vartheta(a)$ system (\ref{wdafddsdjnf}) has a figure eight loop kinking at the fold-fold $O$, which consists of a clockwise rotary, hyperbolic and stable limit cycle in $x\le0$ that graze $x=0$ at a unique point $O$ and its $\mathbb{Z}_2$-symmetric counterpart.

Using the linear change of variables $(x,y)\rightarrow(-y,-x)$, we transform system (\ref{wdafddsdjnf}) into
\begin{equation}\label{wdafddsdjsffnf}
\left(
\begin{array}{c}
\dot x\\
\dot y
\end{array}
\right)=
\left\{
\begin{aligned}
&\left(
\begin{array}{c}
1-y\\
x-(ay+by^3)+\alpha_1+\alpha_2(y-y^2)
\end{array}
\right)\qquad &&{\rm for}\quad y>0,\\
&\left(
\begin{array}{c}
-1-y\\
x-(ay+by^3)-\alpha_1+\alpha_2(y+y^2)
\end{array}
\right)\qquad &&{\rm for}\quad y<0.
\end{aligned}
\right.
\end{equation}
By the analysis of the last paragraph, for fixed $a<0$ and $b=\vartheta(a)$, system $(\ref{wdafddsdjsffnf})$ with $\alpha_1=\alpha_2=0$ has a figure eight loop characterized by {\bf(H1)}, {\bf(H2)} and $\lambda(0)<1$, i.e., the assumptions of Theorem~\ref{mainthm} hold.
Besides, restricted to system (\ref{wdafddsdjsffnf}), the condition (\ref{transversality}) holds because $g^+_{\alpha_1}(0,0;0)=1$, $g^+_{\alpha_2}(0,0;0)=0$ and
$$\kappa_2=\int^{T_0}_0\exp\left(-\int^{T_0}_ta+3by_0^2(s)ds\right)(1-y_0(t))^2y_0(t)dt>0,$$
due to $y_0(t)\ge0$ and $y_0(t)\not\equiv0, 1$, where $(x_0(t),y_0(t))$ is the solution of the unperturbed limit cycle in $y\ge0$ with $(x_0(0),y_0(0))=O$ and $T_0$ is the period.
Consequently, we conclude that there is a codimension-two grazing-sliding bifurcation as described in Theorem~\ref{mainthm}.

%%%%%%%%%%%%%%%%%%%%%%%%%%%%%%%%%%%%%%%%%%%%%%%%%%%%%%%%%%%%%%%%%

\end{document}